\documentclass[11pt]{amsart}

\usepackage{amsmath,amssymb,latexsym,soul,cite,mathrsfs}

\usepackage{color,enumitem,graphicx}
\usepackage[colorlinks=true,urlcolor=blue,
citecolor=red,linkcolor=blue,linktocpage,pdfpagelabels,
bookmarksnumbered,bookmarksopen]{hyperref}
\usepackage[english]{babel}

\usepackage[left=2.9cm,right=2.9cm,top=2.8cm,bottom=2.8cm]{geometry}
\usepackage[hyperpageref]{backref}

\usepackage[colorinlistoftodos]{todonotes}
\makeatletter
\providecommand\@dotsep{5}
\def\listtodoname{List of Todos}
\def\listoftodos{\@starttoc{tdo}\listtodoname}
\makeatother

\numberwithin{equation}{section}
\newtheorem{theorem}{Theorem}[section]
\newtheorem{proposition}[theorem]{Proposition}
\newtheorem{lemma}[theorem]{Lemma}

\begin{document}
	
%	\title {nonsmooth version of fountain theorem}
	
	\title [Heteroclinic solutions for some classes of prescribed mean curvature equations in whole $\mathbb{R}^2$]{Heteroclinic solutions for some classes of prescribed mean curvature equations in whole $\mathbb{R}^2$}

	\author{Claudianor O. Alves}
	\author{Renan J. S. Isneri}

	\address[Claudianor O. Alves]{\newline\indent Unidade Acad\^emica de Matem\'atica
		\newline\indent 
		Universidade Federal de Campina Grande,
		\newline\indent
		58429-970, Campina Grande - PB - Brazil}
	\email{\href{mailto:coalves@mat.ufcg.edu.br}{coalves@mat.ufcg.edu.br}}
	
	\address[Renan J. S. Isneri]
	{\newline\indent Unidade Acad\^emica de Matem\'atica
		\newline\indent 
		Universidade Federal de Campina Grande,
		\newline\indent
		58429-970, Campina Grande - PB - Brazil}
	\email{\href{renan.isneri@academico.ufpb.br}{renan.isneri@academico.ufpb.br}}

	\pretolerance10000

	\begin{abstract}
		\noindent The purpose of this paper consists in using  variational methods to establish the existence of heteroclinic solutions for some classes of prescribed mean curvature equations of the type
		$$
		-div\left(\frac{\nabla u}{\sqrt{1+|\nabla u|^2}}\right) + A(\epsilon x,y)V'(u)=0~~\text{ in }~~\mathbb{R}^2,
		$$
		where $\epsilon>0$ and $V$ is a double-well potential with minima at $t=\alpha$ and $t=\beta$ with $\alpha<\beta$. Here, we consider some class of functions $A(x,y)$ that are oscillatory in the variable $y$ and satisfy different geometric conditions such as periodicity  in all variables or asymptotically periodic at infinity.
	\end{abstract}

	%\thanks{Claudianor Alves was partially supported by CNPq/Brazil Proc. 304036/2013-7 ; Giovany M. Figueiredo was partially
	%supported by  CNPq, Brazil; Gaetano Siciliano  was partially supported by
	%Fapesp and CNPq, Brazil. }

	\maketitle
	
	%------------------------------------------------------------------------------
	\section{Introduction}
	%--------------------------
	
The main goal of this paper is to use variational methods to show the existence of heteroclinic solutions for prescribed mean curvature equation of the type 
$$
-div\left(\frac{\nabla u}{\sqrt{1+|\nabla u|^2}}\right) + A(\epsilon x,y)V'(u)=0~~\text{ in }~~\mathbb{R}^2\eqno(E)
$$ 

\noindent
taking into account different geometric conditions on function $A:\mathbb{R}^2\to\mathbb{R}$ with $\epsilon>0$. Hereafter, we mean by heteroclinic solution a function $u$ that satisfies $(E)$ and has the following asymptotic property at infinity
$$
u(x,y)\rightarrow\alpha\text{ as }x\rightarrow-\infty\text{ and }u(x,y)\rightarrow\beta\text{ as }x\rightarrow+\infty,\text{ uniformly in }y\in\mathbb{R},
$$
where $\alpha, \beta \in \mathbb{R}$ are global minima of the potential  ${V}:\mathbb{R}\to\mathbb{R}$. 

Hereafter, $\mathcal{V}=\{V_{\alpha,\beta}\}$ denotes a set of potentials that satisfy the following properties:  
\begin{itemize}
	\item[($V_1$)] $V_{\alpha,\beta}\in C^1(\mathbb{R},\mathbb{R})$.
	\item[($V_2$)]  $\alpha<\beta$ and $V_{\alpha,\beta}(\alpha)=V_{\alpha,\beta}(\beta)=0$.
	\item[($V_3$)] $V_{\alpha,\beta}(t)\geq 0$ for any $t\in\mathbb{R}$ and $V_{\alpha,\beta}(t)>0$ for all $t\in (\alpha,\beta)$.	
\end{itemize}
Moreover, we assume that the family $\mathcal{V}$ also satisfies the following local uniform estimate involving the numbers $\alpha$ and $\beta$:  
\begin{itemize}
    \item[($V_4$)] Given $\lambda>0$, there exists $M=M(\lambda)>0$ independent of $\alpha, \beta \in (-\lambda,\lambda)$ such that 
    $$ \displaystyle\sup_{t\in(\alpha,\beta)}|V_{\alpha,\beta}'(t)|\leq M, \quad \forall V_{\alpha, \beta} \in \mathcal{V} \quad \text{with} \quad \alpha, \beta \in (-\lambda,\lambda). $$ 
\end{itemize}
The reader is invited to see that we could consider the family $\mathcal{V}$ where the potentials $V_{\alpha, \beta}$ are the form
\begin{equation}\label{v1}
	V_{\alpha,\beta}(t)=(t-\alpha)^2(t-\beta)^2,
\end{equation}
which was inspired by the classical double well Ginzburg-Landau potential. These family is very important, because those potentials arise in the study of heteroclinic solutions for stationary Allen Cahn type equations
$$
-\Delta u+A(x)V'_{\alpha,\beta}(u)=0\quad \text{in}\quad\mathbb{R}^n.
$$
Another important family  $\mathcal{V}$ is related to the Sine-Gordon potential that is formed by potentials of the form    
\begin{equation}\label{v2}
	V_{-\beta,\beta}(t)=\beta+\beta\cos\left(\frac{t\pi}{\beta}\right),
\end{equation}
where $\alpha=-\beta$. For a discussion of physical applications appearing these types of potentials, we refer the interested reader to \cite{Allen,Frenkel}.

In what follows, associated with function $A$ we assume
\begin{itemize}
	\item[($A_1$)] $A$ is continuous and there is $A_0>0$ such that $A(x,y)\geq A_0$ for all $(x,y)\in\mathbb{R}^2$.
	\item[($A_2$)] $A(x,y)=A(x,-y)$ for all $(x,y)\in\mathbb{R}^2$.
	\item[($A_3$)] $A(x,y)=A(x,y+1)$ for any $(x,y)\in\mathbb{R}^2$.
\end{itemize}
Now let us mention the classes of $A$ that we will considered in our work. 
\\
\textbf{Class 1:} $A$ satisfies $(A_1)$-$(A_3)$ and is 1-periodic on the variable $x$.
\\
\textbf{Class 2:} $A$ satisfies $(A_1)$-$(A_3)$ and there exists a continuous function $A_p:\mathbb{R}^2\to\mathbb{R}$, which is 1-periodic on $x$, satisfying $A(x,y)<A_p(x,y)$ for all $(x,y)\in\mathbb{R}^2$ and
$$
|A(x,y)-A_p(x,y)|\to 0\text{ as }|(x,y)|\to+\infty.
$$
\\
\textbf{Class 3:} $A$ satisfies $(A_1)$-$(A_3)$ and
$$
\sup_{y\in[0,1]}A(0,y)<\liminf_{|(x,y)|\rightarrow +\infty}A(x,y)=A_{\infty}<+\infty.
$$
\textbf{Class 4:} $A$ satisfies $(A_2)$-$(A_3)$, is a continuous non-negative function, even in $x$, $A\in L^\infty(\mathbb{R}^2)$ and there exists $K>0$ such that 
$$
\inf_{|x|\geq K,~y\in[0,1]}A(x,y)>0.
$$

We would like to highlight that some of these conditions are well known in the context of the Laplacian operator. For example, a condition like Class 1 was studied by Rabinowitz \cite{Rabinowitz} to show the existence of heteroclinic solution for a class of second order partial differential equations in which he includes the equation of the form
\begin{equation}\label{E1}
	-\Delta u+A(x,y)V'(u)=0~ \text{ in }~\Omega,
\end{equation}
where the set $\Omega$ is a cylindrical domain in $\mathbb{R}^N$ given by $\Omega=\mathbb{R}\times D$ with $D$ being a bounded open set in $\mathbb{R}^{N-1}$ such that $\partial D\in C^1$. In the literature we also find interesting works that study the equation \eqref{E1} in the case that $A(x,y)$ is periodic in all variables when $\Omega=\mathbb{R}^2$, see for example Rabinowitz and Stredulinsky \cite{Rabinowitz1} and Alessio, Gui and Montecchiari \cite{Alessio}. Related to the Classes 2 and 3 we cite a paper by  Alves \cite{Alves0}, where the author established the existence of classical solutions of \eqref{E1} on a cylindrical domain that are heteroclinic in the variable $x$. Finally, the Class 4 was introduced in \cite{Alves2}. 

The prescribed mean curvature operator
\begin{equation}\label{div}
	div\left(\frac{\nabla u}{\sqrt{1+|\nabla u|^2}}\right)
\end{equation}
has been extensively studied in the recent years, due to the close connection with capillarity theory \cite{Finn}. After the pioneering works of Young \cite{Young}, Laplace \cite{Laplace}, and Gauss \cite{Gauss} in the early 18th century about the mean curvature of a capillary surface, much has already been produced in the literature and it is difficult and exhaustive to measure here the vastness of physical applications involving the \eqref{div} operator, however for the interested reader in this subject, we could cite here some problems that appear in optimal transport \cite{Brenier} and in minimal surfaces \cite{Giusti}. Moreover, \eqref{div} also appears in some problems in reaction-diffusion processes which occur frequently in a wide variety of physical and biological settings. For example, in \cite{Kurganov}, Kurganov and Rosenau observed that when the saturation of the diffusion is incorporated into these processes, it may cause a deep impact on the the morphology of the transitions connecting the equilibrium states, as now not only do discontinuous equilibria become permissible, but traveling waves can arise in their place. A specific class of such processes is modeled by the following equation 
\begin{equation}\label{eq}
	u_t=div\left(\frac{\nabla u}{\sqrt{1+|\nabla u|^2}}\right)-aV'(u),
\end{equation}
where the reaction function $V$ is the classical double well Ginzburg-Landau potential and $a$ is a constant. The impact of saturated diffusion on reaction-diffusion processes was investigated by them in the straight line and in the plane.

As indicated in the previous paragraph, \cite{Kurganov} provided a significant physical motivation for the study of equation \eqref{eq} having as main objective the existence of transition-type solutions, that is, entire solutions of \eqref{eq} which are asymptotic in different directions to the equilibrium states of the systems. In this sense, Bonheure, Obersnel and Omari in \cite{Bonheure1} investigated the existence of a heteroclinic solution of the one-dimensional equation
\begin{equation}\label{eq1}
	-\left(\frac{q'}{\sqrt{1+(q')^2}}\right)'+a(t)V'(q)=0~~\text{in}~~\mathbb{R},
\end{equation}
looking for minima of an action functional on a convex subset of $BV_{\text{loc}}(\mathbb{R})$ made of all functions satisfying an asymptotic condition at infinity, where the authors considered as usual $V$ a double-well potential with minima at $t=\pm1$ and $a$ asymptotic to a positive periodic function, that is, $a\in L^{\infty}(\mathbb{R})$ with $0<\displaystyle\operatorname{ess}\inf_{t\in\mathbb{R}}a(t)$ and there is $a^*\in L^{\infty}(\mathbb{R})$ $\tau$-periodic, for some $\tau>0$, such that $a(t)\leq a^*(t)$ almost everywhere on $\mathbb{R}$ satisfying
$$
\operatorname{ess}\!\!\!\!\lim_{|t|\to+\infty}(a^*(t)-a(t))=0.
$$
Recently in \cite{Alves2}, Alves and Isneri also studied the existence of heteroclinic solution for \eqref{eq1}. In that paper, the authors truncated the mean curvature operator to build up a variational framework inspired to the one introduced in \cite{Omari} on Orlicz-Sobolev space $W_{\text{loc}}^{1,\Phi}(\mathbb{R})$ in order to establish the existence of a heteroclinic solution for \eqref{eq1} in the case where the function $a\in L^{\infty}(\mathbb{R})$ is an even non-negative with $0<\displaystyle\inf_{t\geq M}a(t)$ for some $M>0$. Moreover, in the case where the function $a$ is constant, for each initial conditions $q(0)=r_1$ and $q'(0)=r_2$, the uniqueness of the minimal heteroclinic type solutions for \eqref{eq1} has been proved.

Motivated by \cite{Alves2}, \cite{Bonheure1} and \cite{Omari}, in the present paper we intend to analyze the existence of heteroclinic solutions for $(E)$ and their qualitative properties, as well as regularity.  A part of our arguments was inspired by papers due to Rabinowitz \cite{Rabinowitz} and Alves \cite{Alves0}.

The main results of the paper can be now stated in the following form.

\begin{theorem}\label{T1}
	Assume  $\epsilon=1$ and that $A$ belongs to Class 1 or 2. Given $L>0$ there exists $\delta>0$ such that for each $V_{\alpha,\beta} \in \mathcal{V}$ with  $\max\{|\alpha|,|\beta|\}\in(0,\delta)$, the equation $(E)$ with $V=V_{\alpha,\beta}$ possesses a heteroclinic solution $u_{\alpha,\beta}$ from $\alpha$ to $\beta$ in $C^{1,\gamma}_{\text{loc}}(\mathbb{R}^2)$, for some $\gamma\in(0,1)$, satisfying
	\begin{itemize}
		\item [(a)] $u_{\alpha,\beta}$ is 1-periodic on $y$.
		\item [(b)] $\alpha\leq u_{\alpha,\beta}(x,y)\leq \beta$ for any $(x,y)\in\mathbb{R}^2$.  
		\item [(c)] $\|\nabla u_{\alpha,\beta}\|_{L^{\infty}(\mathbb{R}^2)}\leq \sqrt{L}$.
	\end{itemize}
	Moreover, if $V_{\alpha,\beta}\in C^2(\mathbb{R},\mathbb{R})$ then the inequalities in (b) are strict.
\end{theorem}

\begin{theorem}\label{T2}
	Assume that $A$ belongs to Class 3. There is $\epsilon_0>0$ such that for each $\epsilon\in(0,\epsilon_0)$ and $L>0$ given, there exists $\delta>0$ such that for each $V_{\alpha,\beta} \in \mathcal{V}$ with $\max\{|\alpha|,|\beta|\}\in(0,\delta)$, the equation $(E)$ with $V=V_{\alpha,\beta}$ possesses a heteroclinic solution $u_{\alpha,\beta}$ from $\alpha$ to $\beta$ in $C^{1,\gamma}_{\text{loc}}(\mathbb{R}^2)$, for some $\gamma\in(0,1)$, verifying
	\begin{itemize}
		\item [(a)] $u_{\alpha,\beta}$ is 1-periodic on $y$.
		\item [(b)] $\alpha\leq u_{\alpha,\beta}(x,y)\leq \beta$ for any $(x,y)\in\mathbb{R}^2$.  
		\item [(c)] $\|\nabla u_{\alpha,\beta}\|_{L^{\infty}(\mathbb{R}^2)}\leq \sqrt{L}$.
	\end{itemize}
	Moreover, if $V_{\alpha,\beta}\in C^2(\mathbb{R},\mathbb{R})$ occurs then the inequalities in (b) are strict.
\end{theorem}

Demanding a little more of the class $\mathcal{V}$, we can relax the conditions on the function $A$ to ensure the existence of a heteroclinic solution for $(E)$, as the following result says.

\begin{theorem}\label{T3}
		Assume $\alpha=-\beta$, $\epsilon=1$ and that $A$  belonging to Class 4. Then, for each $L>0$ there exists $\delta>0$ such that for each $V_{-\beta,\beta} \in \mathcal{V} \cap C^2(\mathbb{R},\mathbb{R})$ with $\beta\in(0,\delta)$ and satisfying the conditions below 
		\begin{itemize}
			\item [($V_5$)] $V_{-\beta,\beta}''(-\beta),V_{-\beta,\beta}''(\beta)>0$,
			\item[($V_6$)] $V_{-\beta,\beta}(t)=V_{-\beta,\beta}(-t)$ for all $t\in\mathbb{R}$,
		\end{itemize}
		the equation $(E)$ with $V=V_{-\beta,\beta}$ possesses a heteroclinic solution $u_{\beta}$ from $-\beta$ to $\beta$ in $C^{1,\gamma}_{\text{loc}}(\mathbb{R}^2)$, for some $\gamma\in(0,1)$, verifying the following properties:
		\begin{itemize}
			\item [(a)] $u_{\beta}(x,y)=-u_{\beta}(-x,y)$ for any $(x,y)\in\mathbb{R}^2$.
			\item [(b)] $u_{\beta}(x,y)=u_\beta(x,y+1)$ for all $(x,y)\in\mathbb{R}^2$.
			\item [(c)] $0<u_{\beta}(x,y)<\beta$ for $x>0$.  
			\item [(d)] $\|\nabla u_{\beta}\|_{L^{\infty}(\mathbb{R}^2)}\leq \sqrt{L}$.
		\end{itemize}
\end{theorem}

The reader is invited to see that the above theorems are true for the Ginzburg-Landau \eqref{v1} and Sine-Gordon \eqref{v2} potentials when roots $\alpha$ and $\beta$ have a small distance between them.

Motivated by the ideas in \cite{Alves2}, in the proof of the theorems above, we truncate the differential operator involved in $(E)$ in such a way that the new operator can be seen as a quasilinear operator in divergence form. For this reason, as a first step in the present work, we study quasilinear equations of the form
$$
-\Delta_{\Phi}u+A(\epsilon x,y)V'(u)=0~~~\text{in}~~~\mathbb{R}^2,\eqno{(PDE)}
$$
where $\Delta_{\Phi}u=\text{div}(\phi(|\nabla u|)\nabla u)$ and $\Phi:\mathbb{R}\rightarrow [0,+\infty)$ is an $N$-function of the type
\begin{equation}\label{*}
	\Phi(t)=\int_0^{|t|}s\phi(s)ds,	
\end{equation}
where $\phi:(0,+\infty)\rightarrow (0,+\infty)$ is a $C^1$ function verifying the following conditions:
\begin{itemize}
	\item[($\phi_1$)] $\phi(t)>0$ and $(\phi(t)t)'>0$ for any $t>0$.
	\item[($\phi_2$)] There are $l,m\in\mathbb{R}$ with $1<l\leq m$ such that 
	$$
	l-1\leq \dfrac{(\phi(t)t)'}{\phi(t)}\leq m-1,~~\forall t>0.
	$$
	\item[($\phi_3$)] There exist constants $c_1,c_2,\delta>0$ and $q>1$ satisfying
	$$
	c_1t^{q-1}\leq \phi(t)t\leq c_2t^{q-1},~~t\in(0,\delta].
	$$	
\end{itemize}
The solutions of $(PDE)$ are found as minima of the action functional
$$
I(w)=\sum_{j\in\mathbb{Z}}\left(\int_0^1\int_j^{j+1}\left(\Phi(|\nabla w|)+A(\epsilon x,y)V(w)\right)dxdy\right)
$$
on the class of admissible functions
\begin{equation*}
	\Gamma_{\Phi}(\alpha,\beta)\!=\!\left\{w \!\in\! W^{1,\Phi}_{\text{loc}}(\mathbb{R}\times(0,1))\!:\!\tau_kw\to \!\alpha~(\beta)\!\text{ in }\! L^\Phi((0,1)\times(0,1))\!\text{ as }\! k\to\!-\infty~(+\infty)\right\},
\end{equation*}
where $W^{1,\Phi}_{\text{loc}}(\mathbb{R}\times(0,1))$ denotes the usual Orlicz-Sobolev space. We would like to point out that in the study of quasilinear elliptic problems driven by the $\Phi$-Laplacian operator the conditions $(\phi_1)$-$(\phi_2)$ are well-known and guarantee that $\Phi$ is an $N$-function that checks the so called $\Delta_2$-condition (see for instance Appendix \ref{A}). Those conditions ensure that the Orlicz-Sobolev space is reflexive and separable. Moreover, in this work, the conditions $(\phi_1)$-$(\phi_2)$ are enough to show the existence of heteroclinic solution, while assumption $(\phi_3)$ is used to get more information about the behavior of the solution, because it permits to apply a Harnack type inequality found in Trudinger \cite{Trudinger}.

Our results involving the quasilinear problem $(PDE)$ are stated below 

\begin{theorem}\label{T4}
Assume $(\phi_1)$-$(\phi_2)$, $(V_1)$-$(V_3)$, $\epsilon=1$ and that $A$ belongs to Class 1 or 2. Then equation $(PDE)$ has a heteroclinic solution from $\alpha$ to $\beta$ in $C^{1,\gamma}_{\text{loc}}(\mathbb{R}^2)$ for some $\gamma\in(0,1)$ such that 
\begin{itemize}
	\item[(a)] $u(x,y)=u(x,y+1)$ for any $(x,y)\in\mathbb{R}^2$.
	\item[(b)] $\alpha\leq u(x,y)\leq\beta$ for all $(x,y)\in\mathbb{R}^2$.
\end{itemize}
Moreover, taking into account the assumptions $(\phi_3)$ and 
\begin{itemize}
	\item[($V_7$)] There are $d_1,d_2,d_3,d_4>0$ and $\lambda>0$ such that 
	$$
	|V'(t)|\leq d_1\phi(d_2|t-\beta|)|t-\beta|\text{ for all }t\in[\beta-\lambda,\beta+\lambda]
	$$
	and 
	$$
	|V'(t)|\leq d_3\phi(d_4|t-\alpha|)|t-\alpha|\text{ for all }t\in[\alpha-\lambda,\alpha+\lambda],
	$$
\end{itemize}
then the inequalities in (b) are strict. 
\end{theorem}

\begin{theorem}\label{T5}
	Assume $(\phi_1)$-$(\phi_2)$, $(V_1)$-$(V_3)$ and that $A$ belongs to Class 3. Then there is a constant $\epsilon_0>0$ such that for each $\epsilon\in(0,\epsilon_0)$ equation $(PDE)$ has a heteroclinic solution from $\alpha$ to $\beta$ in $C^{1,\gamma}_{\text{loc}}(\mathbb{R}^2)$ for some $\gamma\in(0,1)$ such that 
	\begin{itemize}
		\item[(a)] $u(x,y)=u(x,y+1)$ for any $(x,y)\in\mathbb{R}^2$.
		\item[(b)] $\alpha\leq u(x,y)\leq\beta$ for all $(x,y)\in\mathbb{R}^2$.
	\end{itemize}
	Moreover, assuming $(\phi_3)$ and $(V_7)$ we have that the inequalities in (b) are strict. 
\end{theorem}

\begin{theorem}\label{T6}
	Assume $(\phi_1)$-$(\phi_2)$, $(V_1)$-$(V_3)$ and $(V_6)$ with $\alpha=-\beta$, $\epsilon=1$ and that $A$ belongs to Class 4. Also consider the following assumption
	\begin{itemize}
		\item[($V_8$)] There are $\mu>0$ and $\theta\in(0,\beta)$ such that 
		$$
		\mu\Phi(|t-\beta|)\leq V(t),~~\forall t\in(\beta-\theta,\beta+\theta).
		$$
	\end{itemize}
	Then equation $(PDE)$ possesses a heteroclinic solution $u$ from $-\beta$ to $\beta$ in $C^{1,\gamma}_{\text{loc}}(\mathbb{R}^2)$ for some $\gamma\in(0,1)$ such that
	\begin{itemize}
		\item [(a)] $u(x,y)=-u(-x,y)$ for any $(x,y)\in\mathbb{R}^2$.
		\item [(b)] $u(x,y)=u(x,y+1)$ for all $(x,y)\in\mathbb{R}^2$.
		\item [(c)] $0\leq u(x,y)\leq \beta$ for any $x>0$ and $y\in\mathbb{R}$.  
	\end{itemize}
	Moreover, if $(\phi_3)$ and $(V_7)$ occur then the inequalities in (c) are strict.
\end{theorem}

Here it is worth mentioning that an example of potential $V$ that satisfies the conditions $(V_1)$-$(V_3)$ and $(V_6)$-$(V_8)$ is given by
\begin{equation}\label{v3}
V(t)=\Phi(|(t-\alpha)(t-\beta)|)~\text{ for all }~t\in\mathbb{R},
\end{equation}
where $\Phi$ is an $N$-function of the type \eqref{*} verifying $(\phi_1)$-$(\phi_2)$, which was used in \cite{Alves3}. The reader is invited to see that the condition $(V_7)$ is not necessary to prove the existence of heteroclinic solution from $\alpha$ to $\beta$ for $(PDE)$, however it together with $(\phi_3)$ are used to obtain more information about the behavior of the heteroclinic solution. Moreover, the classical case $\Phi(t)=\frac{t^2}{2}$ corresponds to the Laplacian operator, and in this case, as we are considering a new class of functions $A$, we can rewrite Theorem \ref{T6} as follows

\begin{theorem}\label{T7}
	Assume $V\in C^2(\mathbb{R},\mathbb{R})$,  $(V_2)$-$(V_3)$ and $(V_5)$-$(V_6)$ with $\alpha=-\beta$, and that $A$ belongs to Class 4. Then equation \eqref{E1} with $\Omega=\mathbb{R}^2$ possesses a heteroclinic (classical) solution $u$ from $-\beta$ to $\beta$ such that
	\begin{itemize}
		\item [(a)] $u(x,y)=-u(-x,y)$ for any $(x,y)\in\mathbb{R}^2$.
		\item [(b)] $u(x,y)=u(x,y+1)$ for all $(x,y)\in\mathbb{R}^2$.
		\item [(c)] $0<u(x,y)< \beta$ for any $x>0$ and $y\in\mathbb{R}$.  
	\end{itemize}
\end{theorem}

Before ending this section, we do some comments about our results: First, Theorems \ref{T1}, \ref{T2} and \ref{T3} complement the study carried out in \cite{Alves2} and \cite{Bonheure1}, because in those articles the authors considered the one-dimensional equation \eqref{eq1}, while in the present article we treat $(E)$ in the whole plane and investigated the existence of a heteroclinic solution for $(E)$ for other classes of functions $A$. Moreover, we believe that this is the first article where the existence of a heteroclinic solution for a prescribed mean curvature equation is addressed in whole $\mathbb{R}^2$. Secondly, Theorems \ref{T4} - \ref{T7} hold for all pair of real numbers $(\alpha,\beta)$ with $\alpha<\beta$. Furthermore, in Theorems \ref{T4} and \ref{T5} we can consider a variety of potentials as the prototypes \eqref{v1} and \eqref{v2}, while in Theorem \ref{T6} the potential $V$ must have a strong interaction with the N-function  $\Phi$, see for example \eqref{v3}. Finally, we would also like to point out that Theorems \ref{T4} and \ref{T5} complement the results obtained in \cite{Alves0}, because in that paper the author considered the Laplacian operator while here we considered a large class of quasilinear operators. Moreover, Theorem \ref{T7} brings new contributions to the classic equation \eqref{E1}, since a new class of functions $A$ is considered.

We organize the rest of this work as follows. Motivated by \cite{Alves0,Alves1,Rabinowitz}, in Section \ref{S2} we present the variational framework of the quasilinear problems mentioned above, which will be useful for the next section, and moreover, the proofs of Theorems \ref{T4}, \ref{T5} and \ref{T6} are given. Section \ref{S3} is dedicated to the proofs of Theorems \ref{T1}, \ref{T2} and \ref{T3} that were supported by the study developed in \cite[Section 3]{Alves2}. Finally, we write an Appendix \ref{A} about some results involving Orlicz and Orlicz-Sobolev spaces for unfamiliar readers with the topic.

\section{Existence of heteroclinic solution for quasilinear equations}\label{S2}

The goal of this section is to establish the existence of a heteroclinic solution from $\alpha$ to $\beta$ for $(PDE)$ taking into account the case where $A$ satisfies different geometric conditions. The proof of the existence of solution is given by a minimization argument. To formulate the minimization problem of this section, let us first consider the infinite strip $\Omega=\mathbb{R}\times(0,1)$ of $\mathbb{R}^2$ and for each $j\in\mathbb{Z}$ we define the functional $a_j:W^{1,\Phi}_{\text{loc}}(\Omega)\to \mathbb{R}\cup\{+\infty\}$ by 
$$
a_j(w)=\iint_{\Omega_j}\mathcal{L}(w)dxdy, ~~w\in W^{1,\Phi}_{\text{loc}}(\Omega),
$$
where $\Omega_j=(j,j+1)\times (0,1)$ and 
$$
\mathcal{L}(w)=\Phi(|\nabla w|)+A(\epsilon x,y)V(w).
$$
Under this notation, we also define the energy functional $I:W^{1,\Phi}_{\text{loc}}(\Omega)\to \mathbb{R}\cup\{+\infty\}$ by
$$
I(w)=\sum_{j\in\mathbb{Z}}a_j(w), ~~w\in W^{1,\Phi}_{\text{loc}}(\Omega).
$$
In what follows, for each $k\in\mathbb{Z}$ and $w\in W^{1,\Phi}_{\text{loc}}(\Omega)$ we denote
$$
\tau_kw(x,y)=w(x+k,y)~\text{ for all }~(x,y)\in\Omega.
$$ 
Clearly, $\tau_0w\equiv w$ on $\Omega$. Hereafter, let us identify $\tau_kw|_{\Omega_0}$ with $\tau_kw$ itself. Now, for the purposes of this paper, we will designates by $\Gamma_{\Phi}(\alpha,\beta)$ the class of admissible functions given by 
\begin{equation}\label{Class}
	\Gamma_{\Phi}(\alpha,\beta)\!=\!\left\{w \!\in\! W^{1,\Phi}_{\text{loc}}(\Omega)\!:\!\tau_kw\to \!\alpha\!\text{ in }\! L^\Phi(\Omega_0)\!\text{ as }\! k\to\!-\infty\!\text{ and }\!\tau_kw\to\!\beta\!\text{ in }\! L^\Phi(\Omega_0)\!\text{ as }\! k\to\!+\infty\right\}.
\end{equation} 
We would like to point out that, as $\Phi$ satisfies $\Delta_2$-condition, $\tau_kw$ goes to $\alpha$ in $L^\Phi(\Omega_0)$ as $k$ goes to $-\infty$ if, and only if, 
$$
\iint_{\Omega_k}\Phi(|w-\alpha|)dxdy\to 0\text{ as }k\to-\infty.
$$
Analogously, $\tau_kw$ goes to $\beta$ in $L^\Phi(\Omega_0)$ as $k$ goes to $+\infty$ if, and only if,
$$
\iint_{\Omega_k}\Phi(|w-\beta|)dxdy\to 0\text{ as }k\to+\infty.
$$ 
On the other hand, it is easy to check that the class $\Gamma_{\Phi}(\alpha,\beta)$ is not empty, because the function $\varphi_{\alpha,\beta}:\Omega\to\mathbb{R}$ defined by 
\begin{equation}\label{F1}
	\varphi_{\alpha,\beta}(x,y)=\left\{\begin{array}{lll}
		\beta,&\mbox{if}\quad \beta\leq x&\mbox{and}\quad y\in (0,1),\\ 
		x,&\mbox{if}\quad \alpha\leq x\leq \beta&\mbox{and}\quad y\in (0,1),\\
		\alpha,&\mbox{if}\quad x\leq\alpha&\mbox{and}\quad y\in (0,1)
	\end{array}\right.
\end{equation}
belongs to $\Gamma_{\Phi}(\alpha,\beta)$. By the properties of $\Phi$, $A$ and $V$, 
$$
a_j(w)\geq 0\text{ for all }j\in\mathbb{Z}\text{ and }w\in\Gamma_{\Phi}(\alpha,\beta),
$$
and hence, $I$ is bounded from below on $\Gamma_{\Phi}(\alpha,\beta)$. Furthermore, it is easy to see that the function given in \eqref{F1} has finite energy, that is, $I(\varphi_{\alpha,\beta})<+\infty$, and so, the real number 
$$
c_\Phi(\alpha,\beta)=\inf_{w\in\Gamma_{\Phi}(\alpha,\beta)}I(w)
$$
is well defined. Here it is worth mentioning that we will see throughout this section that critical points of the functional $I$ on the class $\Gamma_{\Phi}(\alpha,\beta)$ are heteroclinic solution from $\alpha$ to $\beta$ for the equation $(PDE)$.

\subsection{The case periodic}\label{Sb1}

In this subsection, we intend to investigate the existence of a heteroclinic solution from $\alpha$ to $\beta$ for $(PDE)$ with $\epsilon=1$ by assuming that $A$ belongs to Class 1 and, unless indicated, the potential $V$ satisfies the assumptions $(V_1)$-$(V_3)$. With the preliminaries contained at the beginning of this section we may state and prove our first result that will be useful in the next lemma.

\begin{lemma}\label{R1}
	If $w\in\Gamma_{\Phi}(\alpha,\beta)$, then for all $k\in\mathbb{Z}$ we have that $\tau_kw\in\Gamma_{\Phi}(\alpha,\beta)$ and $I(\tau_kw)=I(w)$.
\end{lemma}
\begin{proof}
Initially, it is easy to see that $\tau_{k}w\in\Gamma_{\Phi}(\alpha,\beta)$ for any $k\in\mathbb{Z}$ and $w\in\Gamma_{\Phi}(\alpha,\beta)$. On the other hand, for each $j\in\mathbb{Z}$, a simple change variable combined with the periodicity of $A$ in the variable $x$ leads to  
	\begin{equation*}
		\begin{split}
			a_j(\tau_kw)&=\iint_{\Omega_j}\left(\Phi(|\nabla \tau_kw|)+A(x,y)V(\tau_kw)\right)dxdy\\&=\iint_{\Omega_j}\left(\Phi(|\nabla w(x+k,y)|)+A(x+k,y)V(w(x+k,y))\right)dxdy\\&=\iint_{\Omega_{j+k}}\left(\Phi(|\nabla w|)+A(x,y)V(w)\right)dxdy=a_{j+k}(w),
		\end{split}
	\end{equation*}
from which it follows that 
$$
I(\tau_kw)=\sum_{j\in\mathbb{Z}}a_j(\tau_kw)=\sum_{j\in\mathbb{Z}}a_{j+k}(w)=\sum_{j\in\mathbb{Z}}a_j(w)=I(w), 
$$
and the proof is completed.
\end{proof}

Now we employ the Lemma \ref{R1} to prove that the energy functional $I$ reaches the minimum energy in some function of $\Gamma_\Phi(\alpha,\beta)$.

\begin{proposition}\label{R2}
	There exists $u\in\Gamma_\Phi(\alpha,\beta)$ such that $I(u)=c_\Phi(\alpha,\beta)$ and $\alpha\leq u(x,y)\leq \beta$ almost everywhere in $\Omega$.
\end{proposition}
\begin{proof}
Let $(u_n)$ be a minimizing sequence for $I$ on $\Gamma_\Phi(\alpha,\beta)$, that is, $I(u_n)\to c_\Phi(\alpha,\beta)$ as $n\to+\infty$. Thus, there is a constant $M>0$ verifying
\begin{equation}\label{1}
	I(u_n)\leq M~\text{ for all}~n\in\mathbb{N}.	
\end{equation}
We claim that we may assume without loss of generality that the sequence $u_n$ satisfies 
$$
\alpha\leq u_n(x,y)\leq \beta\text{ for all }(x,y)\in \Omega\text{ and }n\in\mathbb{N}.
$$ 
Indeed, just consider
$$
\tilde{u}_n(x,y)=\max\{\alpha,\min\{u_n(x,y),\beta\}\}, ~~(x,y)\in\Omega,
$$ 
instead of $u_n$. Moreover, we can also assume that for each $n\in\mathbb{N}$,
\begin{equation}\label{2}
	\iint_{\Omega_0}\Phi(|u_n-\alpha|)dxdy>\delta  \text{ and }\iint_{\Omega_{k-1}}\Phi(|u_n-\alpha|)dxdy\leq\delta \text{ for }k\leq 0,
\end{equation}
for some $\delta>0$ such that 
\begin{equation}\label{3}
		\delta<\Phi(\beta-\alpha).
\end{equation}
To establish this statement, it suffices to observe that
$$
\liminf_{k\to+\infty}\iint_{\Omega_0}\Phi(|\tau_ku_n-\alpha|)dxdy>0=\lim_{k\to-\infty}\iint_{\Omega_0}\Phi(|\tau_ku_n-\alpha|)dxdy.
$$
Indeed, note first that for each $n\in\mathbb{N}$ fixed, 
$$
\tau_ku_n\to \beta\text{ in }L^\Phi(\Omega_0)\text{ as }k\to +\infty,
$$ 
and by $\Phi\in\Delta_2$ (see for a moment \eqref{NT0}), 
$$
\iint_{\Omega_0}\Phi(|\tau_ku_n-\beta|)dxdy\to 0\text{ as }k\to+\infty.
$$
Thus, since $\alpha\neq\beta$, there are $\delta>0$ and a subsequence of $(\tau_ku_n)_{k\geq0}$, still denoted $(\tau_ku_n)$, such that 
$$
\iint_{\Omega_0}\Phi(|\tau_ku_n-\alpha|)dxdy> \delta,~~\forall k>0.
$$
Note that, without loss of generality, we may assume that $\delta$ satisfies \eqref{3}. On the other hand, using again the fact that $\Phi\in\Delta_2$, $\tau_ku_n$ goes to $\alpha$ in $L^\Phi(\Omega_0)$ as $k$ goes to $-\infty$ implies
$$
\iint_{\Omega_0}\Phi(|\tau_ku_n-\alpha|)dxdy\to 0\text{ as }k\to-\infty,
$$
and so, there exists an integer $\overline{k}_n<0$ such that 
$$
\iint_{\Omega_0}\Phi(|\tau_ku_n-\alpha|)dxdy\leq \delta,~~\forall k\leq \overline{k}_n.
$$
From this, it is possible to find the bigger integer $k_n\in\mathbb{Z}$ such that
$$
\iint_{\Omega_{k-1}}\Phi(|u_n-\alpha|)dxdy\leq\delta\text{ for all }k\leq k_n\text{ and }\iint_{\Omega_{k_n}}\Phi(|u_n-\alpha|)dxdy>\delta,
$$
that is, 
$$
\iint_{\Omega_{j-1}}\Phi(|\tau_{k_n}u_n-\alpha|)dxdy\leq\delta\text{ for all }j\leq 0\text{ and }\iint_{\Omega_0}\Phi(|\tau_{k_n}u_n-\alpha|)dxdy>\delta.
$$
Now, we can apply Lemma \ref{R1} to consider $\tau_{k_n}u_n$ in the place of $u_n$. 

Now, since $\alpha\leq u_n\leq \beta$ in $\Omega$, it is straightforward to check that $(u_n)$ is bounded in $W_{\text{loc}}^{1,\Phi}(\Omega)$. Thereby, in view of Lemma \ref{A4}, $W^{1,\Phi}(K)$ is reﬂexive Banach spaces whenever $K$ is relatively compact in $\Omega$, and so, by a classical diagonal argument, there are a subsequence of $(u_n)$, still denoted by $(u_n)$, and $u\in W_{\text{loc}}^{1,\Phi}(\Omega)$ satisfying 
\begin{equation}\label{9}
	u_n\rightharpoonup u~~\text{in}~~W_{\text{loc}}^{1,\Phi}(\Omega)\text{ as } n\to+\infty,
\end{equation}
\begin{equation}\label{7}
	u_n\to u~~\text{in}~~L^\Phi_{\text{loc}}(\Omega)\text{ as }n\to+\infty
\end{equation}
and
\begin{equation}\label{10}
	u_n(x,y)\rightarrow u(x,y)~~\text{a.e. in}~~\Omega\text{ as } n\to+\infty.
\end{equation}
As a consequence of \eqref{10}, 
\begin{equation}\label{4}
	\alpha\leq u(x,y)\leq \beta\text{ almost everywhere in }\Omega.
\end{equation}
Moreover, from \eqref{1}, we have the inequality below 
$$
\int_{0}^1\int_{-j}^{j}\mathcal{L}(u_n)dxdy\leq M,~~\forall n,j\in\mathbb{N},
$$
which combines with weak lower semicontinuity of $I$ to give 
$$
\int_{0}^1\int_{-j}^{j}\mathcal{L}(u)dxdy\leq M,~~\forall j\in\mathbb{N}.
$$
Therefore, since $j\in\mathbb{N}$ is arbitrary, we conclude that $I(u)\leq M$. With the aid of the previous preliminaries, our goal is to ensure that $u$ belongs to $\Gamma_{\Phi}(\alpha,\beta)$. Towards that end, we will show that
\begin{equation}\label{5}
	\tau_ku\to\alpha\text{ in }L^\Phi(\Omega_0)\text{ as } k\to-\infty.
\end{equation}
To show \eqref{5}, let us consider the sequence $(\tau_ku)_{k\leq 0}$ with $k\in\mathbb{Z}$. Due to Lemma \ref{R1} and the estimate \eqref{4}, it is simple to prove that $(\tau_ku)_{k\leq 0}$ is bounded in $W^{1,\Phi}(\Omega_0)$. Consequently, for some subsequence, there exists $u^*\in W^{1,\Phi}(\Omega_0)$ such that
\begin{equation}\label{6}
	\tau_ku\rightharpoonup u^*\text{ in }W^{1,\Phi}(\Omega_0)~\text{ as }~k\to-\infty,
\end{equation}
\begin{equation}\label{8}
	\tau_ku\to u^*~\text{ in }~L^\Phi(\Omega_0)\text{ as }k\to-\infty
\end{equation}
and 
\begin{equation}\label{11}
	\alpha\leq u^*(x,y)\leq\beta~\text{ almost everywhere on }~\Omega_0.	
\end{equation}
Now, since $I(u)\leq M$, the definition of $I$ ensures that
$$
a_k(u)\to 0\text{ as }|k|\to+\infty.
$$
This together with the periodicity of $A$ yields that 
\begin{equation}\label{77}
	a_0(\tau_k u)\to 0\text{ as }|k|\to+\infty.
\end{equation}
Now, the fact that $a_0$ is weakly lower semicontinuous on $W^{1,\Phi}(\Omega_0)$ and $a_0\geq 0$ together \eqref{6} and \eqref{77} guarantee that $a_0(u^*)=0$. Thereby, \eqref{11} together with the assumptions on functions $A$ and $V$ ensures that $u^*=\alpha$ or $u^*=\beta$ a.e. in $\Omega_0$. On the other hand, it follows from \eqref{2} and \eqref{7} that
\begin{equation*}
	\iint_{\Omega_0}\Phi(|u-\alpha|)dxdy\geq\delta \text{ and }\iint_{\Omega_0}\Phi(|\tau_ku-\alpha|)dxdy\leq\delta \text{ for }k<0.
\end{equation*}
Consequently, taking the limit as $k\to-\infty$ in the
inequality above and employing \eqref{8}, we arrive at  
$$
\iint_{\Omega_0}\Phi(|u^*-\alpha|)dxdy\leq\delta.
$$
From \eqref{3}, one has $u^*=\alpha$ a.e. in $\Omega_0$, showing that the limit \eqref{5} is valid. Now we claim that
\begin{equation}\label{12}
	\tau_ku\to\beta\text{ in }L^{\Phi}(\Omega_0)\text{ as }k\to+\infty.
\end{equation}
Indeed, considering the sequence $(\tau_ku)_{k>0}$ with $k\in\mathbb{N}$, there exist $u^{**}\in W^{1,\Phi}(\Omega_0)$ and a subsequence of $(\tau_ku)$, still denoted $(\tau_ku)$, such that 
\begin{equation}\label{13}
	\tau_ku\rightharpoonup u^{**}\text{ in }W^{1,\Phi}(\Omega_0)\text{ as }k\to+\infty,
\end{equation}
\begin{equation}\label{14}
	\tau_ku\to u^{**}\text{ in }L^\Phi(\Omega_0)\text{ as }k\to+\infty,
\end{equation}
\begin{equation}\label{15}
	\tau_ku\to u^{**}\text{ in }L^1(\Omega_0)\text{ as }k\to+\infty
\end{equation}
and 
\begin{equation}\label{16}
	\tau_ku(x,y)\to u^{**}(x,y)\text{ a.e in }\Omega_0\text{ as }k\to+\infty.
\end{equation}
Arguing as above, we will get that $u^{**}=\alpha$ or $u^{**}=\beta$ a.e in $\Omega_0$. The  claim \eqref{12} follows if we prove that $u^{**}=\beta$ a.e in $\Omega_0$, and to do that, we will split the proof into two steps. So, seeking for a contradiction we assume that $u^{**}=\alpha$ a.e. in $\Omega_0$.\\
\textbf{Step 1:} There are $\epsilon_0>0$ and  $n_1\in\mathbb{N}$ such that 
\begin{equation}\label{17}
	a_{-1}(u_n)+a_0(u_n)=\int_{0}^1\int_{-1}^{1}\left(\Phi(|\nabla u_n|)+A(x,y)V(u_n)\right)dxdy\geq\epsilon_0,~~\forall n\geq n_1. 
\end{equation}

Indeed, if this does not hold, then there is a subsequence $(u_{n_i})$ of $(u_n)$ such that 
$$
\int_{0}^1\int_{-1}^{1}\left(\Phi(|\nabla u_{n_i}|)+A(x,y)V(u_{n_i})\right)dxdy\to0.
$$
Consequently, there is $v\in W^{1,\Phi}((-1,1)\times (0,1))$ such that 
$$
u_{n_i}\rightharpoonup v~~\text{ in }~~W^{1,\Phi}((-1,1)\times (0,1)), \quad u_{n_i}\to v~~\text{ in }~~L^{\Phi}((-1,1)\times (0,1))
$$
and
\begin{equation} \label{Newequation}
v=\alpha \quad \mbox{or} \quad v=\beta \quad \mbox{a.e. in} \quad (-1,1)\times (0,1).
\end{equation}
Making a simple analysis of the estimates contained in \eqref{2}, we infer that \eqref{Newequation} is impossible, which ends this step.
	
To proceed to the next step, let us fix $\tilde{\epsilon}\in(0,\epsilon_0/2)$ and $n_0\geq n_1$ such that
\begin{equation}\label{18}
	I(u_n)\leq c_\Phi(\alpha,\beta)+\frac{\tilde{\epsilon}}{2},~~\forall n\geq n_0,
\end{equation}
where $\epsilon_0$ and $n_1$ were given in Step 1. \\
\textbf{Step 2:} There are $k\in\mathbb{N}$ and $n\geq n_0$ large enough satisfying
\begin{equation}\label{19}
	a_0\left(x(\tau_ku_n-\alpha)+\alpha\right)\leq \frac{\tilde{\epsilon}}{2}.
\end{equation}

In order to show estimate \eqref{19}, we will separately analyze the terms of the functional $a_0$ that will be divided into four parts as follows: \\	
\textbf{Part 1:} There exists $k_0\in\mathbb{N}$ such that for each $k\geq k_0$ there is $n(k)\geq n_0$ verifying
\begin{equation}\label{20}
	\iint_{\Omega_0}A(x,y)V(\tau_{k}u_n)dxdy\leq\frac{\tilde{\epsilon}}{24\cdot4^{m}}, ~~\forall n\geq n(k),
\end{equation}
\begin{equation}\label{21}
	\iint_{\Omega_0}A(x,y)V\left(x(\tau_{k}u_{n}-\alpha)+\alpha\right)dxdy\leq \frac{\tilde{\epsilon}}{4},~~\forall n\geq n(k)
\end{equation}
and 
\begin{equation}\label{22}
	\iint_{\Omega_0}\Phi(|\tau_ku_n-\alpha|)dxdy\leq\frac{\tilde{\epsilon}}{12\cdot4^{m}},~~\forall n\geq n(k),
\end{equation}
where $m$ was given in $(\phi_2)$.

In fact, let us initially note that, since $V\in C^1$ and $\tau_ku(x,y)\in[\alpha,\beta]$ for any $(x,y)\in\Omega_0$, the Mean Value Theorem together with $(V_2)$ gives us
$$
V(\tau_ku),~V(x(\tau_ku-\alpha)+\alpha)\leq R|\tau_ku-\alpha|,~~\forall(x,y)\in\Omega_0,
$$
for some $R>0$. Consequently,
$$
\iint_{\Omega_0}A(x,y)V(x(\tau_ku-\alpha)+\alpha)dxdy\leq R\sup_{\Omega_0}A(x,y)\iint_{\Omega_0}|\tau_{k}u-\alpha|dxdy
$$
and 
$$
\iint_{\Omega_0}A(x,y)V(\tau_ku)dxdy\leq R\sup_{\Omega_0}A(x,y)\iint_{\Omega_0}|\tau_{k}u-\alpha|dxdy.
$$
Now, as we are assuming that $u^{**}=\alpha$ a.e in $\Omega_0$, it follows from \eqref{15} that there is $k_0\in\mathbb{N}$ such that 
\begin{equation}\label{23}
	\iint_{\Omega_0}A(x,y)V(x(\tau_ku-\alpha)+\alpha)dxdy\leq\frac{\tilde{\epsilon}}{8},~~\forall k\geq k_0
\end{equation}
and 
\begin{equation}\label{24}
	\iint_{\Omega_0}A(x,y)V(\tau_ku)dxdy\leq\frac{\tilde{\epsilon}}{48\cdot4^{m}},~~\forall k\geq k_0.
\end{equation}
Furthermore, from \eqref{14}, increasing $k_0$ if necessary, one gets 
\begin{equation}\label{25}
	\iint_{\Omega_0}\Phi(|\tau_ku-\alpha|)dxdy\leq \frac{\tilde{\epsilon}}{24\cdot8^{m}},~~\forall k\geq k_0.
\end{equation}
On the other hand, for each $k\in\mathbb{N}$ fixed, Lebesgue Dominated Convergence Theorem yields
$$
\left|\iint_{\Omega_0}A(x,y)\left(V(x(\tau_ku_n-\alpha)+\alpha)-V(x(\tau_ku-\alpha)+\alpha)\right)dxdy\right|\to0\text{ as }n\to+\infty
$$ 
and 
$$
\left|\iint_{\Omega_0}A(x,y)\left(V(\tau_ku_n)-V(\tau_ku)\right)dxdy\right|\to0\text{ as }n\to+\infty.
$$ 
Moreover, as $\Phi\in\Delta_2$, for each $k\in\mathbb{N}$ we can use the limit \eqref{7} to find 
$$
\iint_{\Omega_0}\Phi(|\tau_ku_n-\tau_ku|)dxdy\to0\text{ as }n\to+\infty.
$$
With everything, for every $k\geq k_0$ there exists $n(k)\geq n_0$ satisfying 
\begin{equation}\label{26}
	\iint_{\Omega_0}A(x,y)V(x(\tau_ku_n-\alpha)+\alpha)dxdy\leq \iint_{\Omega_0}A(x,y)V(x(\tau_{k}u-\alpha)+\alpha)dxdy+\frac{\tilde{\epsilon}}{8},~~\forall n\geq n(k),
\end{equation}
\begin{equation}\label{27}
	\iint_{\Omega_0}A(x,y)V(\tau_ku_n)dxdy\leq \iint_{\Omega_0}A(x,y)V(\tau_{k}u)dxdy+\frac{\tilde{\epsilon}}{48\cdot4^{m}},~~\forall n\geq n(k),
\end{equation}
and 
\begin{equation}\label{28}
	\iint_{\Omega_0}\Phi(|\tau_ku_n-\tau_ku|)dxdy\leq\frac{\tilde{\epsilon}}{24\cdot8^{m}},~~\forall n\geq n(k).
\end{equation}
Finally, analyzing all the estimates, a direct use of \eqref{23}-\eqref{24} and \eqref{26}-\eqref{27} lead to \eqref{20} and \eqref{21}. To see the inequality \eqref{22}, note that from Lemma \ref{A5}-$(a)$ one gets
$$
\Phi(|\tau_ku_n-\alpha|)\leq 2^{m}\left(\Phi(|\tau_ku_n-\tau_ku|)+\Phi(|\tau_ku-\alpha|)\right). 
$$
Now, \eqref{22} follows from  \eqref{25} and \eqref{28}, finishing the first part. \\
\textbf{Part 2:} There are $k\geq k_0$ and $n\geq n(k)$ such that 
\begin{equation}\label{29}
	a_0\left(\tau_{k}u_n\right)\leq \frac{\tilde{\epsilon}}{12\cdot4^{m}}.
\end{equation}

If the estimate above does not occur, for any $k\geq k_0$ one has
$$
a_0\left(\tau_{k}u_j\right)> \frac{\tilde{\epsilon}}{12\cdot4^{m}},~~\forall n\geq n(k).
$$
%Then, by definition of $a_0$,
%$$
%\iint_{\Omega_0}\Phi(|\nabla(\tau_ku_j)|)dxdy\geq \frac{\tilde{\epsilon}}{12\cdot4^{m}}-\iint_{\Omega_0}A(x,y)V(\tau_{k}u_j)dxdy,~~\forall j\geq j(k),
%$$
%which combines with \eqref{20} to give
%$$
%\iint_{\Omega_0}\Phi(|\nabla(\tau_ku_j)|)dxdy\geq \frac{\tilde{\epsilon}}{24\cdot4^{m}},~~\forall j\geq j(k).
%$$
Now let $p\in\mathbb{N}$ be such that 
$$
(p+1)\frac{\tilde{\epsilon}}{12\cdot4^{m}}>M,
$$
where $M$ was given in \eqref{1}. Fixing $i\in\mathbb{N}$ such that $i>\max\{n(k):k_0\leq k\leq k_0+p\}$ we have
$$
I(u_i)\geq \sum_{t=k_0}^{k_0+p}a_t(u_i)=\sum_{t=k_0}^{k_0+p}a_0(\tau_tu_i)\geq (p+1)\frac{\tilde{\epsilon}}{12\cdot4^{m}}>M,
$$
which contradicts \eqref{1}, showing \eqref{29}.\\
\textbf{Part 3:} For $k$ and $n$ as in Part 2, one has 
\begin{equation}\label{30}
	\iint_{\Omega_0}\Phi(|\partial_x(\tau_ku_n)|)dxdy,~\iint_{\Omega_0}\Phi(|\partial_y(\tau_ku_n)|)dxdy\leq\frac{\tilde{\epsilon}}{12\cdot4^{m}}.
\end{equation}

Indeed,	just notice that the inequality at \eqref{29} together with the facts that $\Phi$ is increasing on $[0,+\infty)$ and 
$$
|\partial_x(\tau_ku_n)|, |\partial_y(\tau_ku_n)|\leq |\nabla(\tau_ku_n)|,
$$ 
leads to estimate \eqref{30}.\\
\textbf{Part 4:} For $k$ and $n$ as in Part 2, one has
\begin{equation}\label{31}
	\iint_{\Omega_0}\Phi(|\nabla (x(\tau_ku_n-\alpha)+\alpha)|)dxdy\leq\frac{\tilde{\epsilon}}{4}.
\end{equation}

To show the estimate \eqref{31}, we first observe that
$$
\partial_x\left(x(\tau_ku_n-\alpha)+\alpha\right)=\tau_ku_n-\alpha+x\partial_x(\tau_ku_n)
$$
and 
$$
\partial_y\left(x(\tau_ku_n-\alpha)+\alpha\right)=x\partial_y\left(\tau_ku_n\right).
$$
Therefore, from Lemma \ref{A5}-$(a)$, 
$$
\Phi(|\nabla (x(\tau_ku_n-\alpha)+\alpha)|)\leq 4^{m}\left(\Phi(|\partial_x(\tau_ku_n)|)+\Phi(|\tau_ku_n-\alpha|)+\Phi(|\partial_y(\tau_ku_n)|)\right)\text{ on }\Omega_0,
$$
and the Part 4 follows from Parts 1 and 3.

Finally, the estimate \eqref{19} contained in Step 2 is immediately verified from \eqref{21} and \eqref{31}. We are now ready to use Steps 1 and 2 to complete the proof of Claim \eqref{12}. To this end, fix $k$ and $n$ as in Step 2 and define the following function
$$ 
U_n(x,y)=\left\{\begin{array}{lll}
		\alpha,&\mbox{if}\quad x\leq k&\mbox{and}\quad y\in (0,1),\\ 
		(u_n(x,y)-\alpha)(x-k)+\alpha,&\mbox{if}\quad k\leq x\leq k+1&\mbox{and}\quad y\in (0,1),\\
		u_n(x,y),&\mbox{if}\quad x>k+1&\mbox{and}\quad y\in (0,1).
	\end{array}\right.
$$
So, it is clear that $U_n\in\Gamma_{\Phi}(\alpha,\beta)$ and 
$$
a_k(U_n)=a_0(x(\tau_ku_n-\alpha)+\alpha).
$$
Hence, 
$$
c_{\Phi}(\alpha,\beta)\leq I(U_n)=a_k(U_n)+\sum_{j=k+1}^{+\infty}a_j(u_n)\leq a_k(U_n)+I(u_n)-a_{0}(u_n)-a_{-1}(u_n),
$$	
that is, 
$$
c_{\Phi}(\alpha,\beta)\leq a_0(x(\tau_ku_n-\alpha)+\alpha)+I(u_n)-a_{0}(u_n)-a_{-1}(u_n).
$$
Invoking estimates \eqref{17}, \eqref{18} and \eqref{19}, one gets
$$
c_\Phi(\alpha,\beta)\leq \frac{\tilde{\epsilon}}{2}+c_\Phi(\alpha,\beta)+\frac{\tilde{\epsilon}}{2}-\epsilon_0=\tilde{\epsilon}+c_\Phi(\alpha,\beta)-\epsilon_0<c_\Phi(\alpha,\beta)-\frac{\epsilon_0}{2},
$$
which is absurd. Therefore, \eqref{12} occurs and $u\in\Gamma_\Phi(\alpha,\beta)$. To conclude the proof, it remains to show that $I(u)=c_\Phi(\alpha,\beta)$. For this purpose, given $\epsilon>0$ there exists $n_0\in\mathbb{N}$ such that 
$$
\sum_{-j}^{j}a_k(u_n)\leq c_\Phi(\alpha,\beta)+\epsilon,~\forall n\geq n_0\text{ and }\forall j\in\mathbb{N}.
$$
Letting $n\to+\infty$ and after $j\to+\infty$, we find
$$
I(u)\leq c_\Phi(\alpha,\beta)+\epsilon.
$$
Since $\epsilon$ is arbitrary, we derive that $I(u)=c_\Phi(\alpha,\beta)$, and the proof is completed.
\end{proof}

In order to find a periodic solution $u(x,y)$ in the variable $y$ for the equation $(PDE)$, we will consider the following class
$$
K_{\Phi}(\alpha,\beta)=\left\{ u\in\Gamma_\Phi(\alpha,\beta):I(u)=c_\Phi(\alpha,\beta),u(x,0)=u(x,1)\text{ in }\mathbb{R},\alpha\leq u\leq\beta\text{ a.e. on }\Omega\right\}. 
$$
Next, we are going to show that $K_{\Phi}(\alpha,\beta)$ is not empty. 

\begin{lemma}\label{R3}
	It holds that $K_{\Phi}(\alpha,\beta)\neq\emptyset$. 
\end{lemma}
\begin{proof}
Initially, for each $w\in\Gamma_\Phi(\alpha,\beta)$ we set
$$
I_1(w)=\iint_{\mathbb{R}\times(0,\frac{1}{2})}\mathcal{L}(w)dxdy~~\text{and}~~I_2(w)=\iint_{\mathbb{R}\times(\frac{1}{2},1)}\mathcal{L}(w)dxdy.
$$ 	
Now, choosing $u\in \Gamma_\Phi(\alpha,\beta)$ as in Proposition \ref{R2}, we can write 
\begin{equation}\label{32}
	I(u)=I_1(u)+I_2(u)=c_{\Phi}(\alpha,\beta).
\end{equation}
Suppose for a moment that $I_1(u)\leq I_2(u)$  holds. Then considering the function 
$$
v(x,y)=\left\{\begin{array}{lll}
	u(x,y),&\mbox{if}\quad x\in\mathbb{R}&\mbox{and}\quad 0\leq y\leq \frac{1}{2},\\
	u(x,1-y),&\mbox{if}\quad x\in\mathbb{R} &\mbox{and}\quad \frac{1}{2}\leq y\leq 1,
\end{array}\right.
$$
it is clear that  $v\in\Gamma_\Phi(\alpha,\beta)$ and thanks to the assumptions $(A_2)$-$(A_3)$ a straightforward computation gives
\begin{equation}\label{33}
	I_2(v)=I_1(v)=I_1(u).
\end{equation}
According to \eqref{32} and \eqref{33},
$$
c_\Phi(\alpha,\beta)\leq I(v)=I_1(v)+I_2(v)\leq I(u)=c_\Phi(\alpha,\beta),
$$
from where it follows that $I(v)=c_\Phi(\alpha,\beta)$ with $v(x,0)=v(x,1)$ for every $x\in\mathbb{R}$ and $v(x,y)\in[\alpha,\beta]$ a.e. in $\Omega$. On the other hand, if $I_2(u)\leq I_1(u)$ occurs then in this case we define the function  
$$
\tilde{v}(x,y)=\left\{\begin{array}{lll}
	u(x,1-y),&\mbox{if}\quad x\in\mathbb{R}&\mbox{and}\quad 0\leq y\leq \frac{1}{2}\\
	u(x,y),&\mbox{if}\quad x\in\mathbb{R} &\mbox{and}\quad \frac{1}{2}\leq y\leq 1.
\end{array}\right.
$$
Consequently, using the same ideas discussed just above, we obtain that $\tilde{v}\in\Gamma_\Phi(\alpha,\beta)$ with $I_1(\tilde{v})=I_2(\tilde{v})=I_2(u)$, from where it follows that $I(\tilde{v})=c_\Phi(\alpha,\beta)$, $\tilde{v}(x,0)=\tilde{v}(x,1)$ for any $x\in\mathbb{R}$ and $\alpha\leq \tilde{v}\leq\beta$ a.e in $\Omega$, which completes the proof.
\end{proof}

We would like to emphasize here that the functions of $K_\Phi(\alpha,\beta)$ can be extended periodically in the variable $y$ on $\mathbb{R}^2$ with period 1. For this reason, it will be convenient to assume that the elements of $K_\Phi(\alpha,\beta)$ are extended to the whole real plane. Finally, we will show our best result of this section.

\begin{lemma}\label{R4}
If $u\in K_\Phi(\alpha,\beta)$, then $u$ is a weak solution of $(PDE)$ with $\epsilon=1$. Moreover, $u$ is a heteroclinic solution from $\alpha$ to $\beta$ which belongs to $C^{1,\gamma}_{\text{loc}}(\mathbb{R}^2)$ for some $\gamma\in (0,1)$. 
\end{lemma}
\begin{proof}
Initially, let be $u\in K_\Phi(\alpha,\beta)$ and $\psi \in C^{\infty}_0(\mathbb{R}^2)$. Then, the function $v=u+t\psi \in \Gamma_\Phi(\alpha,\beta)$ and 
\begin{equation*}\label{34}
	\iint_{\Omega}\left(\phi(|\nabla u|)\nabla u\nabla\psi+A(x,y)V'(u)\psi\right)dxdy=0.
\end{equation*}
The equality above allows us to use the same arguments found in the proof \cite[Theorem 1.1]{Alves1} to prove that 
\begin{equation*}
	\iint_{\mathbb{R}^2}\left(\phi(|\nabla u|)\nabla u\nabla\psi+A(x,y)V'(u)\psi\right)dxdy=0,~~\forall\psi\in C^{\infty}_0(\mathbb{R}^2),	
\end{equation*} 
from where it follows that $u$ is a weak solution for $(PDE)$ with $\epsilon=1$. The assumption $(\phi_2)$ permits to apply a well known regularity result developed by Lieberman \cite[Theorem 1.7]{Lieberman} to conclude that $u\in C^{1,\gamma}_{\text{loc}}(\mathbb{R}^2)$ for some $\gamma\in(0,1)$. Moreover, similar to the proof of \cite[Theorem 1.1]{Alves1}, we also have that $u$ is a heteroclinic solution from $\alpha$ to $\beta$, that is, 
$$
u(x,y)\to\alpha\text{ as }x\to-\infty\text{ and }u(x,y)\to\beta\text{ as }x\to+\infty,\text{ uniformly in }y\in\mathbb{R},
$$
and the lemma follows.
\end{proof}

Now, we will show our last result in this subsection, which ends the study of the equation $(PDE)$ in the case where $A$ is periodic in all variables.

\begin{lemma}\label{R4a}
	 Assume $(\phi_3)$ and $(V_7)$. Then, if $u\in K_\Phi(\alpha,\beta)$ we have that
	 $$
	 \alpha<u(x,y)<\beta\text{ for all }(x,y)\in\mathbb{R}^2.
	 $$ 
\end{lemma}
\begin{proof}
Let $u\in K_\Phi(\alpha,\beta)$ and observe that $\alpha\leq u\leq \beta$ on $\mathbb{R}^2$. In what follows, we will show that $u(x,y)<\beta$ for any $(x,y)\in\mathbb{R}^2$. Indeed, assume for the sake of contradiction that there exists $(x_0,y_0)\in\mathbb{R}^2$ such that $u(x_0,y_0)=\beta$. Therefore, by the geometry of $u$, we can consider a compact set $\mathcal{O}$ contained in $\mathbb{R}^2$ such that there exists $(x_1,y_1)\in\mathcal{O}$ with $u(x_1,y_1)<\beta$. Having that in mind, setting the function $\tilde{\phi}:(0,+\infty)\rightarrow (0,+\infty)$ by
	$$
	\tilde{\phi}(t)=\left\{\begin{array}{ll}
		\phi(t),&\mbox{if}\quad t\in(0,R],\\
		\dfrac{\phi(R)t^{q-2}}{R^{q-2}}, &\mbox{if}\quad t\in(R,+\infty),
	\end{array}\right.
	$$ 
	where $R>\max\{\|\nabla u\|_{L^{\infty}(\mathcal{O})},\delta\}$ and the constants $q$ and $\delta$ were given in $(\phi_3)$, a direct computation implies that there are positive real numbers $\gamma_1$ and $\gamma_2$, which dependent on $\delta$, $q$, $R$, $c_1$ and $c_2$, such that 
	\begin{equation*}\label{36}
		\tilde{\phi}(t)t\leq\gamma_1t^{q-1}~\text{ and }~\tilde{\phi}(t)t^2\geq\gamma_2 t^q~\text{ for all }~ t\geq 0.
	\end{equation*} 
	Using the function $\tilde{\phi}$, let us define the vector measurable function  $G:\mathbb{R}^2\times\mathbb{R}\times\mathbb{R}^2\to\mathbb{R}^2$  by   
	$$
	G(z,t,p)=\frac{\tilde{\phi}(|p|)p}{\gamma_2},
	$$ 
	which satisfies  
	$$
	|G(z,t,p)|\leq \frac{\gamma_1}{\gamma_2}|p|^{q-1}~\text{ and }~
	pG(z,t,p)\geq |p|^q~\text{ for all }~ (z,t,p)\in\mathbb{R}^2\times\mathbb{R}\times\mathbb{R}^2.
	$$ 
	Furthermore, we will also consider the scalar measurable function $B:\mathbb{R}^2\times\mathbb{R}\times\mathbb{R}^2\to\mathbb{R}$ given by  
	$$
	B(z,t,p)=\frac{A(z)V'(\beta-t)}{\gamma_2}.
	$$
	Now, combining $(\phi_3)$ with $(V_7)$, it is possible to ensure that for each $M>0$ there exists $C_M>0$ satisfying 
	$$
	|B(z,t,p)|\leq C_M|t|^{q-1}~~\text{ for all }~~(z,t,p)\in\mathbb{R}^2\times(-M,M)\times\mathbb{R}^2.
	$$ 
	All these information are necessary to guarantee that $G$ and $B$ fulfill the structure required in the Harnack type inequality found in Trudinger \cite[Theorem 1.1]{Trudinger}. So, setting $v(z)=\beta-u(z)$ for $z\in\mathbb{R}^2$, we infer that $v$ is a weak solution of the quasilinear equation 
	$$
	div ~G(z,v,\nabla v)+B(z,v,\nabla v)=0~~\text{in}~~\mathcal{O}. 
	$$ 
Employing \cite[Theorem 1.1]{Trudinger}, we deduce that $v=0$ on $\mathcal{O}$, that is, $u =\beta$ on $\mathcal{O}$, which contradicts the fact that $(x_1,y_1)\in\mathcal{O}$ with $u(x_1,y_1)<\beta$. Likewise, we can apply a similar argument to show that $u(x,y)>\alpha$ for any $(x,y)\in\mathbb{R}^2$, and hence the proof is completed.
\end{proof}

\subsection{The case asymptotic at infinity to a periodic function}\label{Sb2}

In this subsection we will study the existence of a heteroclinic solution for $(PDE)$ with $\epsilon=1$ and $A$ belongs to Class 2, that is, $A$ is asymptotic at infinity to a periodic function $A_p$. Moreover, unless otherwise indicated, we will consider here the conditions $(\phi_1)$-$(\phi_2)$ on $\phi$ and $(V_1)$-$(V_3)$ on $V$. The fact that we are assuming that the function $A$ is only asymptotically periodic with respect to $x$ brings a lot of difficulties and some arguments explored in the periodic case do not work anymore. 

In this section, let us consider the functional $I_p:W^{1,\Phi}_{\text{loc}}(\Omega)\to\mathbb{R}\cup\{+\infty\}$ by 
$$
I_p(w)=\sum_{j\in\mathbb{Z}}a_{p,j}(w),~~w\in W^{1,\Phi}_{\text{loc}}(\Omega),
$$
where 
$$
a_{p,j}(w)=\iint_{\Omega_j}\left(\Phi(|\nabla w|)+A_p(x,y)V(w)\right)dxdy.
$$
Moreover, we use $c_{p,\Phi}(\alpha,\beta)$ to denote the real number given by 
$$
c_{p,\Phi}(\alpha,\beta)=\inf_{w\in\Gamma_\Phi(\alpha,\beta)}I_p(w).
$$
From Subsection \ref{Sb1}, we know that there is $w_0\in\Gamma_\Phi(\alpha,\beta)$ such that 
$I_p(w_0)=c_{p,\Phi}(\alpha,\beta)$, and so,
\begin{equation}\label{001}
	c_\Phi(\alpha,\beta)\leq I(w_0)<I_p(w_0)=c_{p,\Phi}(\alpha,\beta).
\end{equation}
The inequality \eqref{001} establishes an important relation between $c_\Phi(\alpha,\beta)$ and $c_{p,\Phi}(\alpha,\beta)$, which will be useful to achieve the objective of this subsection. With these information, we are ready to prove the main result of this subsection.

\begin{proposition}\label{R5}
	There is $u\in\Gamma_\Phi(\alpha,\beta)$ such that $I(u)=c_\Phi(\alpha,\beta)$ satisfying $\alpha\leq u(x,y)\leq\beta$ almost everywhere in $\Omega$.
\end{proposition}
\begin{proof}
First of all, note that there exists a minimizing sequence $(u_n)\subset\Gamma_\Phi(\alpha,\beta)$ for $I$ satisfying  
$$
\alpha\leq u_n(x,y)\leq\beta,~\forall n\in\mathbb{N}\text{ and }(x,y)\in\Omega.
$$
Moreover, there are $u\in W^{1,\Phi}_{\text{loc}}(\Omega)$ and a subsequence of $(u_n)$, still denoted by $(u_n)$, such that
\begin{equation}\label{37}
	u_n\rightharpoonup u~~\text{in}~~W_{\text{loc}}^{1,\Phi}(\Omega),
\end{equation}
\begin{equation}\label{38}
	u_n\to u~~\text{in}~~L^\Phi_{\text{loc}}(\Omega)
\end{equation}
and
\begin{equation}\label{39}
	u_n(x,y)\rightarrow u(x,y)\text{ a.e. in } \Omega.
\end{equation}
From \eqref{37}-\eqref{39},   
\begin{equation}\label{003}
	I(u)\leq c_\Phi(\alpha,\beta)
\end{equation}
and
$$
\alpha\leq u(x,y)\leq \beta\text{ a.e. in } \quad \Omega.
$$ 
Now our goal is to show that $u\in\Gamma_\Phi(\alpha,\beta)$. To achieve this goal, similar to the proof of Proposition \ref{R2}, we have that $(\tau_{k}u)_{k>0}$ is a bounded sequence in $W^{1,\Phi}(\Omega_0)$. Thereby, for some subsequence of $(\tau_ku)$, still denoted by itself, there is $u^*\in W^{1,\Phi}(\Omega_0)$ such that  
$$
\tau_ku\rightharpoonup u^*\text{ in }W^{1,\Phi}(\Omega_0)~\text{ as }~k\to+\infty,
$$
$$
\tau_ku\to u^*~\text{ in }~L^\Phi(\Omega_0)\text{ as }k\to+\infty
$$
and 
$$
\tau_ku(x,y)\to u^*(x,y)\text{ a.e on }\Omega_0\text{ as }k\to+\infty.
$$
We claim that $u^*= \alpha$ or $u^*= \beta$ a.e. in $\Omega_0$. Indeed, since $I(u)\leq c_\Phi(\alpha,\beta)$, we infer that $a_k(u)$ goes to 0 as $k$ goes to $+\infty$, and so, by change of variable, 
$$
\iint_{\Omega_0}\left(\Phi(|\nabla \tau_ku|)+A(x+k,y)V(\tau_{k}u)\right)dxdy\to0\text{ as }k\to+\infty.
$$
Consequently, by $(A_1)$,
$$
\iint_{\Omega_0}\left(\Phi(|\nabla \tau_ku|)+A_0V(\tau_{k}u)\right)dxdy\to0\text{ as }k\to+\infty,
$$
from where it follows that
$$
\iint_{\Omega_0}\left(\Phi(|\nabla u^*|)+A_0V(u^*)\right)dxdy=0.
$$
By the assumptions on $\Phi$ and $V$, we derive that $u^*= \alpha$ or $u^*= \beta$ a.e. in $\Omega_0$. Next, we claim that 
\begin{equation}\label{40}
	u^*= \beta\text{ a.e in } \Omega_0.
\end{equation}
To establish the claim above, let us assume by contradiction that $u^*= \alpha$ a.e. in $\Omega_0$. So, as a consequence, we will prove that given $\delta\in (0,\Phi(\beta-\alpha))$ there are $(u_{n_i})\subset (u_n)$, $(k_i)\subset\mathbb{N}$ and $i_*\in\mathbb{N}$ such that 
\begin{equation}\label{41}
	i_*<k_i\text{ for all }i\in\mathbb{N},  k_i\to+\infty\text{ and }n_i\to+\infty\text{ as }i\to+\infty,
\end{equation}
\begin{equation}\label{42}
	\iint_{\Omega_0}\Phi(|\tau_{j}u_{n_i}-\alpha|)dxdy<\delta\text{ and }\iint_{\Omega_0}\Phi(|\tau_{k_i}u_{n_i}-\alpha|)dxdy\geq\delta~\forall j\in[i_*,k_i-1]\cap\mathbb{N}.
\end{equation} 
Indeed, since $\tau_{k}u$ goes to $\alpha$ in $L^{\Phi}(\Omega_0)$ as $k$ goes to $+\infty$, given $\delta\in(0,\Phi(\beta-\alpha))$, there is $i_*=i_*(\delta)\in\mathbb{N}$ satisfying 
\begin{equation}\label{44}
	\iint_{\Omega_0}\Phi(|\tau_{k}u-\alpha|)dxdy<\frac{\delta}{2^{m+1}}~~\forall k\geq i_*.
\end{equation}
In particular, 
\begin{equation}\label{43}
	\iint_{\Omega_0}\Phi(|\tau_{i_*}u-\alpha|)dxdy<\frac{\delta}{2^{m+1}}.
\end{equation}
Gathering \eqref{38} and \eqref{43} with the fact that $\Phi\in\Delta_2$, we find $n_1\in\mathbb{N}$ satisfying 
$$
\iint_{\Omega_0}\Phi(|\tau_{i_*}u_{n_1}-\alpha|)dxdy<\delta.
$$
Thereby, since $u_{n_1}\in\Gamma_\Phi(\alpha,\beta)$, we may fix $k_1\geq i_*+1$ as the first natural number such that
$$
\iint_{\Omega_0}\Phi(|\tau_{j}u_{n_1}-\alpha|)dxdy<\delta\text{ and }\iint_{\Omega_0}\Phi(|\tau_{k_1}u_{n_1}-\alpha|)dxdy\geq\delta~~\forall j\in[i_*,k_1-1]\cap\mathbb{N}.
$$
On the other hand, according to \eqref{44},
$$
\iint_{\Omega_0}\Phi(|\tau_{i_*}u-\alpha|)dxdy,~\iint_{\Omega_0}\Phi(|\tau_{i_*+1}u-\alpha|)dxdy<\frac{\delta}{2^{m+1}}.
$$
Hence, in the same manner we can see that there is $n_2\in\mathbb{N}$ such that $n_2>n_1$ and 
$$
\iint_{\Omega_0}\Phi(|\tau_{i_*}u_{n_2}-\alpha|)dxdy,~\iint_{\Omega_0}\Phi(|\tau_{i_*+1}u_{n_2}-\alpha|)dxdy<\delta.
$$
Using the fact that $u_{n_2}\in\Gamma_\Phi(\alpha,\beta)$, we can find $k_2\geq i_*+2$ as the first natural number satisfying 
$$
\iint_{\Omega_0}\Phi(|\tau_{j}u_{n_2}-\alpha|)dxdy<\delta\text{ and }\iint_{\Omega_0}\Phi(|\tau_{k_2}u_{n_2}-\alpha|)dxdy\geq\delta,~~\forall j\in[i_*,k_2-1]\cap\mathbb{N}.
$$
Repeating the above argument, there are sequences $(u_{n_i})\subset (u_n)$ and $(k_i)\subset\mathbb{N}$ such that $k_i\geq i_*+i$ satisfying \eqref{41} and \eqref{42}. So, for some subsequence, there is $w\in W^{1,\Phi}_{\text{loc}}(\Omega)$ such that
\begin{equation}\label{45}
	\tau_{k_i}u_{n_i}\rightharpoonup w\text{ in }W^{1,\Phi}_{\text{loc}}(\Omega)~\text{ as }~i\to+\infty,
\end{equation}
\begin{equation}\label{46}
	\tau_{k_i}u_{n_i}\to w~\text{ in }~L^\Phi_{\text{loc}}(\Omega)\text{ as }i\to+\infty,
\end{equation}
\begin{equation}\label{47}
	\tau_{k_i}u_{n_i}(x,y)\to w(x,y)\text{ a.e. in }\Omega\text{ as }i\to+\infty
\end{equation}
and
\begin{equation}\label{48}
	\alpha\leq w(x,y)\leq \beta\text{ a.e. in }\Omega.
\end{equation} 
Now, setting the functional  
$$
a_j^i(v)=\iint_{\Omega_j}\left(\Phi(|\nabla v|)+A(x+k_i,y)V(v)\right)dxdy,~~v\in W^{1,\Phi}_{\text{loc}}(\Omega),~i\in\mathbb{N},~j\in\mathbb{Z},
$$
a simple change of variables gives us
$$
a_j^i(\tau_{k_i}u_{n_i})=a_{j+k_i}(u_{n_i}),
$$
and so, 
\begin{equation}\label{49}
	\sum_{j\in\mathbb{Z}}a_j^i(\tau_{k_i}u_{n_i})=\sum_{j\in\mathbb{Z}}a_{j}(u_{n_i})=I(u_{n_i}),~~\forall i\in\mathbb{N}.
\end{equation}
From \eqref{49}, one has 
\begin{equation}\label{50}
	a_{p,0}(\tau_kw)\to 0~\text{ as }~|k|\to+\infty. 
\end{equation}
To see this, it suffices to show that $I_p(w)\leq c_\Phi(\alpha,\beta)$. Indeed, combining the fact that $A(x+k_i,y)$ goes to $A_p(x,y)$ as $i$ goes to $+\infty$ with \eqref{47}, one gets  
$$
A(x+k_i,y)V(\tau_{k_i}u_{n_i}(x,y))\to A_p(x,y)V(w(x,y))\text{ a.e. in }\Omega.
$$  	
Therefore, the Fatou’s Lemma and \eqref{45} provide 
$$
\sum_{-j}^{j}a_{p,j}(w)\leq \liminf_{i\to+\infty}\sum_{-j}^{j}a_j^i(\tau_{k_i}u_{n_i}),~~\forall j\in\mathbb{N}.
$$
As $j$ is arbitrary, \eqref{49} guarantees that
\begin{equation}\label{002}
	I_p(w)\leq \liminf_{i\to+\infty}I(u_{n_i})=c_\Phi(\alpha,\beta),
\end{equation}
and \eqref{50} is proved. Thereby, passing
to a subsequence if necessary, a direct computation shows that
$$
\tau_kw\to \alpha\text{ or }\beta\text{ in }L^\Phi(\Omega_0)\text{ as }k\to-\infty
$$
and 
$$
\tau_kw\to \alpha\text{ or }\beta\text{ in }L^\Phi(\Omega_0)\text{ as }k\to+\infty.
$$
Our goal now is to ensure that $w\in\Gamma_\Phi(\alpha,\beta)$.
\\ \textbf{Claim 1:} $\tau_kw\to \alpha\text{ in }L^\Phi(\Omega_0)\text{ as }k\to-\infty.$

Indeed, note first that for each $j\in\mathbb{N}$, there is $i_0=i_0(j)\in\mathbb{N}$ such that $k_i-1\geq k_i-j\geq i_*$ for all $i\geq i_0$, where $i_*\in\mathbb{N}$ was given in \eqref{41}. According to \eqref{42},
$$
\iint_{\Omega_0}\Phi(|\tau_{k_i-j}u_{n_i}-\alpha|)dxdy<\delta,~~\forall j\in\mathbb{N},
$$
that is, 
$$
\iint_{\Omega_{-j}}\Phi(|\tau_{k_i}u_{n_i}-\alpha|)dxdy<\delta,~~\forall j\in\mathbb{N}.
$$
Invoking \eqref{46}, we can increase $i$ if necessary to obtain
$$
\iint_{\Omega_0}\Phi(|\tau_{-j}w-\alpha|)dxdy\leq\delta,~\forall j\in\mathbb{N}.
$$
Our claim is proved by noting that $\delta\in(0,\Phi(\beta-\alpha))$. 
\\
\textbf{Claim 2:} $\tau_kw\to\beta\text{ in }L^\Phi(\Omega_0)\text{ as }k\to+\infty.$

Assume by contradiction that $\tau_kw\to\alpha\text{ in }L^\Phi(\Omega_0)\text{ as }k\to+\infty.$ Let us break down the proof of Claim 2 into two steps.
\\ 
\textbf{Step 1:} There are $\epsilon_0>0$ and  $i_0\in\mathbb{N}$ such that 
\begin{equation}\label{52}
	\int_{k_i-1}^{k_i+1}\int_0^1\left(\Phi(|\nabla u_{n_i}|)+V(u_{n_i})\right)dxdy\geq\frac{\epsilon_0}{\tilde{A}},~~\forall i\geq i_0,
\end{equation}
where $\tilde{A}=\min\{1,A_0\}$.

Indeed, if this does not occur, then there is a subsequence $(\tau_{k_{i_j}}u_{n_{i_j}})$ of $(\tau_{k_i}u_{n_i})$ such that 
$$
\int_{-1}^{1}\int_0^1\left(\Phi(|\nabla \tau_{k_{i_j}}u_{n_{i_j}}|)+V(\tau_{k_{i_j}}u_{n_i})\right)dxdy\to0\text{ as }j\to+\infty.
$$
Recalling that $(u_n)$ is bounded in $W^{1,\Phi}_{\text{loc}}(\Omega)$, then going to a subsequence if necessary, there exists $v\in W^{1,\Phi}((-1,1)\times(0,1))$ such that 
\begin{equation}\label{51*}
\tau_{k_{i_j}}u_{n_{i_j}}\rightharpoonup v\text{ in }W^{1,\Phi}((-1,1)\times (0,1))\text{ and }\tau_{k_{i_j}}u_{n_{i_j}}\to v\text{ in }L^{\Phi}((-1,1)\times (0,1)).
\end{equation}
From the assumptions on $\Phi$ and $V$, we have $v=\alpha$ or $v=\beta$ a.e. in $(-1,1)\times(0,1)$. On the other hand, from \eqref{42}, 
\begin{equation}\label{51}
	\iint_{\Omega_{-1}}\Phi(|\tau_{k_i}u_{n_i}-\alpha|)dxdy<\delta\text{ and }\iint_{\Omega_0}\Phi(|\tau_{k_i}u_{n_i}-\alpha|)dxdy\geq\delta,~~\forall i\in\mathbb{N}.
\end{equation}
Finally, taking the limit of $k_i \to +\infty$ in  \eqref{51} and using the limit \eqref{51*} we find a contradiction, finishing the proof of Step 1. 

In what follows, fixing $\epsilon \in (0,\frac{\epsilon_0}{2})$ and increasing $i_0$ if necessary, we obtain  
\begin{equation}\label{53}
	I(u_{n_i})\leq c_\Phi(\alpha,\beta)+\frac{\epsilon}{4},~~\forall i\geq i_0.
\end{equation}
\\
\textbf{Step 2:} There exist $j\in\mathbb{N}$ and $i\geq i_0$ large enough satisfying 
	\begin{equation}\label{54}
		a_{p,j}\left((\tau_{k_i}u_{n_i}-\alpha)(x-j)+\alpha\right)\leq \frac{\epsilon}{2}.
\end{equation}

The proof of Step 2 follows as in the proof of Proposition \ref{R2}, and so, it will be omitted. In the sequel, let us consider $j\in\mathbb{N}$ and $i\geq i_0$ as in Step 2. Setting the function
$$ 
U_{j,i}(x,y)=\left\{\begin{array}{lll}
	\alpha,&\mbox{if}\quad x\leq j&\mbox{and}\quad y\in (0,1),\\ 
	(\tau_{k_i}u_{n_i}(x,y)-\alpha)(x-j)+\alpha,&\mbox{if}\quad j\leq x\leq j+1&\mbox{and}\quad y\in (0,1),\\
	\tau_{k_i}u_{n_i}(x,y),&\mbox{if}\quad x>j+1&\mbox{and}\quad y\in (0,1),
\end{array}\right.
$$
it is simple to check that $U_{j,i}\in\Gamma_\Phi(\alpha,\beta)$ and
\begin{equation}\label{55}
	c_{p,\Phi}(\alpha,\beta)\leq I_p(U_{j,i})=a_{p,j}(U_{j,i})+\sum_{t=j+1}^{+\infty}a_{p,t}(\tau_{k_i}u_{n_i})=a_{p,j}(U_{j,i})+\sum_{t=j+1+k_i}^{+\infty}a_{p,t}(u_{n_i}).
\end{equation}
We claim that increasing $i_0$ if necessary, one gets 
\begin{equation}\label{56}
	\sum_{t=j+1+k_i}^{+\infty}a_{p,t}(u_{n_i})\leq \sum_{t=j+1+k_i}^{+\infty}a_{t}(u_{n_i})+\frac{\epsilon}{4}. 
\end{equation}
Indeed, since the function $A$ belongs to Class 2, we infer that there is $R>0$ such that
$$
A_p(x,y)-A(x,y)\leq \frac{A_0\epsilon}{4C},~~\forall |x|\geq R\text{ and }\forall y\in(0,1),
$$
where $C>0$ is a constant such that $I(u_n)\leq C$ for all $n\in\mathbb{N}$. Consequently, 
$$
\int_{R}^{+\infty}\int_0^1\left(A_p(x,y)-A(x,y)\right)V(u_{n_i})dxdy\leq\frac{A_0\epsilon}{4C}\int_{R}^{+\infty}\int_0^1V(u_{n_i})dxdy\leq\frac{\epsilon}{4},
$$
and therefore, increasing $i$ if necessary the last inequality is sufficient to justify \eqref{56}. In view of \eqref{55} and \eqref{56}, one has 
\begin{equation}\label{57}
	c_{p,\Phi}(\alpha,\beta)\leq a_{p,j}(U_{j,i})+\sum_{t=j+1+k_i}^{+\infty}a_{t}(u_{n_i})+\frac{\epsilon}{4}.
\end{equation}
On the other hand, according to Step 1,  
$$
\sum_{t=j+1+k_i}^{+\infty}a_{t}(u_{n_i})+\epsilon_0\leq \sum_{t=j+1+k_i}^{+\infty}a_{t}(u_{n_i})+\tilde{A}\int_{k_i-1}^{k_i+1}\int_0^1\left(\Phi(|\nabla u_{n_i}|)+V(u_{n_i})\right)dxdy\leq I(u_{n_i}),
$$
which together with \eqref{57} yields that
$$
c_{p,\Phi}(\alpha,\beta)\leq a_{p,j}(U_{j,i})+I(u_{n_i})-\epsilon_0+\frac{\epsilon}{4}.
$$
This together with \eqref{53} leads to
$$
c_{p,\Phi}(\alpha,\beta)\leq a_{p,j}(U_{j,i})+c_\Phi(\alpha,\beta)+\frac{\epsilon}{2}-\epsilon_0.
$$
Recalling that $\epsilon \in (0,\frac{\epsilon_0}{2})$ and using \eqref{54}, we arrive at
$$
c_{p,\Phi}(\alpha,\beta)\leq c_\Phi(\alpha,\beta)+\epsilon-\epsilon_0<c_\Phi(\alpha,\beta)-\frac{\epsilon_0}{2},
$$
contradicting \eqref{001}. This proves the Claim 2.

Finally, by virtue of Claims 1 and 2, we infer that $w\in\Gamma_\Phi(\alpha,\beta)$. Furthermore,  from \eqref{002} we also have 
$$
c_{p,\Phi}(\alpha,\beta)\leq I_p(w)\leq c_\Phi(\alpha,\beta),
$$
obtaining a new contradiction, and our claim \eqref{40} is proved. As a byproduct, 
\begin{equation}\label{571}
	\tau_ku\to\beta~\text{ in }~L^\Phi(\Omega_0)\text{ as }k\to+\infty.
\end{equation}
A similar argument works to prove that
\begin{equation}\label{572}
	\tau_ku\to\alpha~\text{ in }~L^\Phi(\Omega_0)\text{ as }k\to-\infty.
\end{equation}
Combining \eqref{571} and \eqref{572} with \eqref{003} we get precisely the assertion of the proposition.
\end{proof}

Considering here $K_\Phi(\alpha,\beta)$ as described in Subsection \ref{Sb1}, the same argument explored in that subsection guarantees that $K_\Phi(\alpha,\beta)$ is a non-empty set and allows us to write the following result.

\begin{lemma}\label{R5a}
	There exists a weak solution $u$ of $(PDE)$ with $\epsilon=1$, such that $u \in K_\Phi(\alpha,\beta)\cap C^{1,\gamma}_{\text{loc}}(\mathbb{R}^2)$ for some $\gamma\in (0,1)$ and it is heteroclinic solution from $\alpha$ to $\beta$ and 1-periodic in $y$. Moreover, under conditions $(\phi_3)$ and $(V_7)$ we have that 
	$$
	\alpha<u(x,y)<\beta\text{ for all }(x,y)\in\mathbb{R}^2.
	$$
\end{lemma}

%%%%%%%%%%%%%%%%%%%%%%%%%%%%%%%%%%%%%%%%%%%%%%%%%%%%%%%

\subsection{The proof of Theorem \ref{T4}}\label{Sb3}
This theorem is an immediate consequence of the Lemmas \ref{R4}, \ref{R4a} and \ref{R5a}.  

%%%%%%%%%%%%%%%%%%%%%%%%%%%%%%%%%%%%%%%%%%%%%%%%%%%%%%%

\subsection{The proof of Theorem \ref{T5}}\label{Sb4}

The goal of this section is to establish the proof of Theorem \ref{T5}. In this specific case, we are considering that the function $A$ belongs to Class 3. In \cite{Alves0}, Alves called this class of {\it Rabinowitz’s condition}, because an assumption like that has been introduced by Rabinowitz \cite[Theorem 4.33]{Rabinowitz2} to build up a variational framework to study the existence of solution for a partial differential equation of the type
$$
-\epsilon^2\Delta u+A(x)u=f(u)~\text{ in }~\mathbb{R}^N,
$$
where $\epsilon>0$, $f:\mathbb{R}\to\mathbb{R}$ is a continuous function with subcritical growth and $A:\mathbb{R}^N\to\mathbb{R}$ is a continuous
function satisfying
$$
0<\inf_{x\in\mathbb{R}^N}A(x)<\liminf_{|x|\to+\infty} A(x).
$$
Now we will mainly focus on some preliminary results that are crucial in our approach. As a beginning, let us denote by   $I_{\epsilon},I_{\infty}:W^{1,\Phi}_{\text{loc}}(\Omega)\to\mathbb{R}\cup \{+\infty\}$ the following functionals 
$$
I_{\epsilon}(v)=\sum_{j\in\mathbb{Z}}a_{\epsilon,j}(v)~\text{ and }~I_{\infty}(v)=\sum_{j\in\mathbb{Z}}a_{\infty,j}(v),
$$
where  
$$
a_{\epsilon,j}(v)=\iint_{\Omega_{j}}\left(\Phi(|\nabla v|)+A(\epsilon x,y)V(v)\right)dxdy
$$ 
and
$$
a_{\infty,j}(v)=\iint_{\Omega_{j}}\left(\Phi(|\nabla v|)+A_{\infty}V(v)\right)dxdy.
$$ 
Moreover, we indicate by $c_{\epsilon,\Phi}(\alpha,\beta)$ and $c_{\infty,\Phi}(\alpha,\beta)$ the real numbers  
$$
c_{\epsilon,\Phi}(\alpha,\beta)=\inf_{v\in\Gamma_\Phi(\alpha,\beta)} I_\epsilon(v)~\text{ and }~c_{\infty,\Phi}(\alpha,\beta)=\inf_{v\in\Gamma_\Phi(\alpha,\beta)}I_\infty(v).
$$ 

Here we would like to emphasize that throughout this subsection the potential $V$ satisfies the conditions $(V_1)$-$(V_3)$. The next lemma establishes an important relation between the real numbers $c_{\epsilon,\Phi}(\alpha,\beta)$ and $c_{\infty,\Phi}(\alpha,\beta)$, which will play an essential rule in our approach. 

\begin{lemma}\label{R6}
	According to the notation above,  
	$$
	\limsup_{\epsilon\rightarrow 0^+}c_{\epsilon,\Phi}(\alpha,\beta)<c_{\infty,\Phi}(\alpha,\beta).
	$$
\end{lemma}
\begin{proof}
The proof is similar to that discussed in the proof of \cite[Lemma 4.1]{Alves0} and its proof is omitted. 
\end{proof}

We are now ready to prove the following result.

\begin{proposition}\label{R7}
	There exists $\epsilon_0>0$ such that for each $\epsilon\in(0,\epsilon_0)$ there is $u_\epsilon\in\Gamma_\Phi(\alpha,\beta)$ satisfying $I_\epsilon(u_\epsilon)=c_{\epsilon,\Phi}(\alpha,\beta)$ and $\alpha\leq u_\epsilon(x,y)\leq\beta$ almost everywhere $(x,y)\in\Omega$.
\end{proposition}
\begin{proof}
The idea here is to use a variant of the proof of Proposition \ref{R5} to establish the proposition. First of all, thanks to Lemma \ref{R6} we may fix $\epsilon_0>0$ small enough verifying
\begin{equation}\label{66}
	c_{\epsilon,\Phi}(\alpha,\beta)<c_{\infty,\Phi}(\alpha,\beta),~~\forall \epsilon\in (0,\epsilon_{0}).
\end{equation}
Now, arguing as in Subsection \ref{Sb1}, for each $\epsilon\in(0,\epsilon_0)$ there exist a minimizing sequence $(u_n)\subset\Gamma_\Phi(\alpha,\beta)$ for $I_\epsilon$ and $u_\epsilon\in W^{1,\Phi}_{\text{loc}}(\Omega)$ such that
$$
\alpha\leq u_n(x,y)\leq\beta,~\forall (x,y)\in\Omega~\text{and}~\forall n\in\mathbb{N},
$$ 
$$
u_n\rightharpoonup u_\epsilon\text{ in }W^{1,\Phi}_{\text{loc}}(\Omega),
$$
$$
u_n\rightarrow u_\epsilon\text{ in }L^{\Phi}_{\text{loc}}(\Omega),
$$
$$
u_n(x,y)\rightarrow u_\epsilon(x,y)~~\text{a.e in }\Omega,
$$ 
$$
\alpha\leq u_\epsilon(x,y)\leq\beta~~\text{a.e in } \Omega
$$
and
\begin{equation}\label{67}
	I_{\epsilon}(u_\epsilon)\leq c_{\epsilon,\Phi}(\alpha,\beta).
\end{equation}
By a similar argument  to the one used in the proof of Proposition \ref{R5}, there are $u^*_{\epsilon}\in W^{1,\Phi}(\Omega_0)$ and a subsequence of $(\tau_{k}u_{\epsilon})$, still denoted by itself, such that 
$$
\tau_ku_\epsilon\rightharpoonup u^*_\epsilon\text{ in }W^{1,\Phi}(\Omega_0)~\text{ as }~k\to+\infty,
$$
$$
\tau_ku_\epsilon\to u^*_\epsilon~\text{ in }~L^\Phi(\Omega_0)\text{ as }k\to+\infty
$$
and 
$$
\tau_ku_\epsilon(x,y)\to u^*_\epsilon(x,y)\text{ a.e in }\Omega_0\text{ as }k\to+\infty,
$$
where $u^*_\epsilon= \alpha$ or $u^*_\epsilon= \beta$ a.e. in $\Omega_0$. As in the proofs of Propositions \ref{R2} and \ref{R5}, we want to show that $u_\epsilon\in\Gamma_\Phi(\alpha,\beta)$. Toward that end, we show that $u^*_\epsilon= \beta$ a.e. in $\Omega_0$. The argument is similar to that developed in Proposition \ref{R5}, but we present the proof in detail for the reader’s convenience. Indeed, arguing by contradiction, assume that $u^*_\epsilon=\alpha$ a.e. in $\Omega_0$. Thus, given $\delta\in (0,\Phi(\beta-\alpha))$ there exist $i_*\in\mathbb{N}$, a sequence $(k_i)\subset\mathbb{N}$ and a subsequence $(u_{n_i})$ of $(u_n)$ such that $i_*<k_i$ for all $i\in\mathbb{N}$, $k_i\to+\infty$ and $n_i\to+\infty$ as $i\to+\infty$ and 
\begin{equation*}
	\iint_{\Omega_0}\Phi(|\tau_{j}u_{n_i}-\alpha|)dxdy<\delta\text{ and }\iint_{\Omega_0}\Phi(|\tau_{k_i}u_{n_i}-\alpha|)dxdy\geq\delta~\forall j\in[i_*,k_i-1]\cap\mathbb{N}.
\end{equation*} 
Consequently, considering the sequence $(\tau_{k_i}u_{n_i})$, for some subsequence, there exists $w_\epsilon\in W^{1,\Phi}_{\text{loc}}(\Omega)$ satisfying
\begin{equation*}
	\tau_{k_i}u_{n_i}\rightharpoonup w_\epsilon\text{ in }W^{1,\Phi}_{\text{loc}}(\Omega)~\text{ as }~i\to+\infty,
\end{equation*}
\begin{equation*}
	\tau_{k_i}u_{n_i}\to w_\epsilon~\text{ in }~L^\Phi_{\text{loc}}(\Omega)\text{ as }i\to+\infty
\end{equation*}
and
\begin{equation*}
	\alpha\leq w_\epsilon(x,y)\leq \beta\text{ a.e. in }\Omega.
\end{equation*} 
Setting the functional 
$$
a^i_{\epsilon,j}(v)=\iint_{\Omega_j}\left(\Phi(|\nabla v|)+A(\epsilon x+\epsilon k_i,y)V(v)\right)dxdy,~~v\in W^{1,\Phi}_{\text{loc}}(\Omega),~i\in\mathbb{N}\text{ and }j\in\mathbb{Z},
$$
it is easy to check that  
\begin{equation*}
	\sum_{j\in\mathbb{Z}}a_{\epsilon,j}^i(\tau_{k_i}u_{n_i})=\sum_{j\in\mathbb{Z}}a_{\epsilon,j}(u_{n_i})=I_\epsilon(u_{n_i}),~~\forall i\in\mathbb{N}.
\end{equation*}
This fact together with the limit below 
$$
\liminf_{i\to+\infty}A(\epsilon x+\epsilon k_i,y)=A_\infty
$$  
implies that
\begin{equation}\label{59}
	I_\infty(w_\epsilon)\leq c_{\epsilon,\Phi}(\alpha,\beta),	
\end{equation}
and so, $a_{\infty,j}(w_\epsilon)$ goes to $0$ as $j$ goes to $\pm\infty$. So, by passing to a subsequence if necessary, it is easy to see that
$$
\tau_kw_\epsilon\to \alpha\text{ or }\beta\text{ in }L^\Phi(\Omega_0)\text{ as }k\to\pm\infty.
$$
The same ideas explored in the proof of Claim 1 of Proposition \ref{R5} ensures that
$$
\tau_kw_\epsilon\to\alpha\text{ in }L^\Phi(\Omega_0)\text{ as }k\to-\infty.
$$
Next, we are going to prove that  
\begin{equation}\label{60}
	\tau_kw_\epsilon\to\beta\text{ in }L^\Phi(\Omega_0)\text{ as }k\to+\infty.
\end{equation}
Assume for contradiction that \eqref{60} is not true. Arguing as in the proof of Proposition \ref{R5}, it follows that there are $\tilde{\epsilon}_0>0$, $j\in\mathbb{N}$ and $i\in\mathbb{N}$ large enough such that for some fixed $\tilde{\epsilon}\in(0,\frac{\tilde{\epsilon}_0}{2})$ one has
\begin{equation}\label{61}
	\int_{k_i-1}^{k_i+1}\int_0^1\left(\Phi(|\nabla u_{n_i}|)+V(u_{n_i})\right)dxdy\geq\frac{\tilde{\epsilon}_0}{\tilde{A}},
\end{equation}
\begin{equation}\label{62}
	I_\epsilon(u_{n_i})\leq c_{\epsilon,\Phi}(\alpha,\beta)+\frac{\tilde{\epsilon}}{4}
\end{equation}
and 
\begin{equation}\label{63}
	a_{\infty,j}\left((\tau_{k_i}u_{n_i}-\alpha)(x-j)+\alpha\right)\leq \frac{\tilde{\epsilon}}{2}.
\end{equation}
Using the function $U_{j,i}\in\Gamma_\Phi(\alpha,\beta)$ given by
$$ 
U_{j,i}(x,y)=\left\{\begin{array}{lll}
	\alpha,&\mbox{if}\quad x\leq j&\mbox{and}\quad y\in (0,1),\\ 
	(\tau_{k_i}u_{n_i}(x,y)-\alpha)(x-j)+\alpha,&\mbox{if}\quad j\leq x\leq j+1&\mbox{and}\quad y\in (0,1),\\
	\tau_{k_i}u_{n_i}(x,y),&\mbox{if}\quad x>j+1&\mbox{and}\quad y\in (0,1),
\end{array}\right.
$$
we derive that
\begin{equation}\label{64}
	c_{\infty,\Phi}(\alpha,\beta)\leq I_\infty(U_{j,i})=a_{\infty,j}(U_{j,i})+\sum_{t=j+1+k_i}^{+\infty}a_{\infty,t}(u_{n_i}).
\end{equation}
Now, since the function $A$ belongs to Class 3, increasing $i$ if necessary, an easy computation shows that 
\begin{equation}\label{65}
	\sum_{t=j+1+k_i}^{+\infty}a_{\infty,t}(u_{n_i})\leq \sum_{t=j+1+k_i}^{+\infty}a_{\epsilon,t}(u_{n_i})+\frac{\tilde{\epsilon}}{4}. 
\end{equation}
Thus, from \eqref{61}-\eqref{65}, 
$$
c_{\infty,\Phi}(\alpha,\beta)\leq a_{\infty,j}(U_{j,i})+I_\epsilon(u_{n_i})-\epsilon_0+\frac{\tilde{\epsilon}}{4}\leq c_{\epsilon,\Phi}(\alpha,\beta)-\frac{\tilde{\epsilon}_0}{2},
$$
contrary to \eqref{66}. Therefore, $w_\epsilon\in\Gamma_\Phi(\alpha,\beta)$ and \eqref{59} leads to
$$
c_{\infty,\Phi}(\alpha,\beta)\leq I_\infty(w_\epsilon)\leq c_{\epsilon,\Phi}(\alpha,\beta),
$$
which again contradicts \eqref{66}. Consequently, we conclude from the study carried out here that $\tau_ku_{\epsilon}\to\beta$ in $L^\Phi(\Omega_0)$ as $k\to+\infty$. By a similar argument, we can conclude that $u_\epsilon\in\Gamma_\Phi(\alpha,\beta)$ for $\epsilon\in(0,\epsilon_0)$. Moreover, by \eqref{67}, we must have $I_\epsilon(w_\epsilon)= c_{\epsilon,\Phi}(\alpha,\beta)$, finishing the proof. 
\end{proof}

Finally, we can now prove our main result of this subsection.

\noindent{\textbf{Proof of Theorem \ref{T5}.}}\\
Initially, we will consider the following set
\begin{equation}\label{80}
	K_{\epsilon,\Phi}(\alpha,\beta)=\left\{u\!\in\!\Gamma_\Phi(\alpha,\beta)\!:\!I_\epsilon(u)\!=\!c_{\epsilon,\Phi}(\alpha,\beta),~u(x,0)=u(x,1)\text{ in }\mathbb{R},~\alpha\leq u\leq\beta\text{ a.e. on }\Omega \right\},
\end{equation}
which consists of points of minima of $I_\epsilon$ on $\Gamma_\Phi(\alpha,\beta)$ that are seen as functions defined on $\mathbb{R}^2$ being 1-periodic on the variable $y$. Next, from Proposition \ref{R7} we can proceed analogously to the proof of Lemma \ref{R3} for show that $K_{\epsilon,\Phi}(\alpha,\beta)$ is non empty whenever $\epsilon\in(0,\epsilon_0)$. Finally, we point out that the Theorem \ref{T5} follows following the same steps of Subsection \ref{Sb1} and the details are left to the reader. \hspace{12.2 cm} $\Box$

\subsection{The proof of Theorem \ref{T6}}\label{Sb5}
We exhibit in this subsection the proof of Theorem \ref{T6}. To build a framework for this theorem and avoid some bothersome technicalities, we will always assume here that the potential $V$ satisfies the conditions $(V_1)$-$(V_3)$, $(V_6)$ and $(V_8)$. Furthermore, we will consider assumptions $(\phi_1)$-$(\phi_2)$ on $\phi$, $\epsilon=1$ and that $A$ belongs to Class 4. Next we consider the following class of admissible functions
$$
\Gamma_{\Phi}^o(\beta)=\left\{ v\in\Gamma_{\Phi}(-\beta,\beta):v(x,y)=-v(-x,y)\text{ a.e. in }\Omega\text{ and }0\leq v(x,y)\leq \beta \text{ for a.e }x\geq 0 \right\}
$$
and the real number
$$
c_{\Phi}^o(\beta)=\inf_{v\in\Gamma_{\Phi}^o(\beta)}I(v),
$$
where $\Gamma_{\Phi}(-\beta,\beta)$ is given as in \eqref{Class}. Now it is important to point out that $\Gamma_{\Phi}^o(\beta)$ is not empty, because the function $\varphi_{-\beta,\beta}$ defined as in \eqref{F1} belongs to $\Gamma_{\Phi}^o(\beta)$ with $I(\varphi_{-\beta,\beta})<+\infty$. Having said that, we shall now explore the conditions $(V_6)$ and $(V_8)$ to show that the following class
\begin{equation}\label{79}
	K_\Phi^o(\beta)=\left\{v\in\Gamma_{\Phi}^o(\beta):I(v)=c_{\Phi}^o(\beta)\text{ and }v(x,0)=v(x,1)\text{ in }\mathbb{R}\right\}	
\end{equation}
is not empty. Hereafter, we will assume that the functions of $K_\Phi^o(\beta)$ are periodically extended in $\mathbb{R}^2$ on the variable $y$. Therefore, $K_\Phi^o(\beta)$ is constituted by (minimal) heteroclinic type solutions of $(PDE)$ with $\epsilon=1$ that are 1-periodic in $y$ and odd in $x$. 

\begin{lemma}\label{R8}
	It holds that $K_\Phi^o(\beta)\neq\emptyset$.
\end{lemma}
\begin{proof}
By some standard computations, one easily verifies that there exists a minimizing sequence $(u_n)\subset\Gamma_{\Phi}^o(\beta)$ for $I$ such that 
$$
0\leq u_n(x,y)\leq\beta,~\forall n\in\mathbb{N}\text{ and }x\geq 0.
$$	
Besides that, there exist $u\in W^{1,\Phi}_{\text{loc}}(\Omega)$ and a subsequence of $(u_n)$, still denoted by $(u_n)$, satisfying
\begin{equation}\label{68}
	u_n\rightharpoonup u~~\text{in}~~W_{\text{loc}}^{1,\Phi}(\Omega),
\end{equation}
\begin{equation}\label{69}
	u_n\to u~~\text{in}~~L^\Phi_{\text{loc}}(\Omega)
\end{equation}
and
\begin{equation}\label{70}
	u_n(x,y)\rightarrow u(x,y)~~\text{a.e. on}~~\Omega.
\end{equation}
We conclude from \eqref{70} that $u(x,y)=-u(-x,y)$ almost everywhere $(x,y)\in\Omega$ and $0\leq u(x,y)\leq \beta$ for almost every $x\geq0$, and finally by \eqref{68}-\eqref{69} it is easy to check that
\begin{equation}\label{71}
	I(u)\leq c_\Phi^o(\beta).
\end{equation}
Now we claim that $u\in\Gamma_{\Phi}^o(\beta)$. To establish our claim, we assume for the sake of contradiction that $\tau_{k}u$ not goes to $\beta$ as $k$ goes to $+\infty$ in $L^\Phi(\Omega_0)$. Thereby, since $\Phi\in\Delta_2$ there are $\epsilon>0$ and a subsequence $(k_{i})$ of natural numbers with $k_i\to+\infty$ such that
\begin{equation}\label{72}
	\iint_{\Omega_0}\Phi\left(|\tau_{k_i}u-\beta|\right)dxdy\geq\epsilon,~~~\forall i\in\mathbb{N}.
\end{equation} 
On the other hand, $(V_1)$-$(V_3)$ and $(V_8)$ yield 
$$
\tilde{\mu}\Phi(|t-\beta|)\leq V(t),~~\forall t\in[0,\beta],
$$
for some $\tilde{\mu}>0$.
Consequently, 
$$
I(u)\geq\sum_{i\in\mathbb{N}}\left(\iint_{\Omega_{k_i}}A(x,y)V(u)dxdy\right)\geq \tilde{\mu}\sum_{i\in\mathbb{N}}\left( \iint_{\Omega_{k_i}}A(x,y)\Phi\left(|u-\beta|\right)dxdy\right),
$$
that is, 
$$
I(u)\geq\tilde{\mu}\sum_{i\in\mathbb{N}}\left( \iint_{\Omega_0}A(x+k_i,y)\Phi\left(|\tau_{k_i}u-\beta|\right)dxdy\right).
$$
Now, fixing $i_0\in\mathbb{N}$ such that $|x+k_i|\geq K$ for any $x\in[0,1]$ and $i\geq i_0$,  the fact that $A$ belongs to Class 4 leads to
$$
I(u)\geq\tilde{\mu}a_0\sum_{i\geq i_0}\left( \iint_{\Omega_0}\Phi\left(|\tau_{k_i}u-\beta|\right)dxdy\right),
$$
where 
$$
a_0=\inf_{|x|\geq K,~y\in[0,1]}A(x,y)>0.
$$
Hence, by \eqref{72}, 
$$
I(u)\geq a_0\tilde{\mu}\sum_{i\geq i_0}\epsilon=+\infty,
$$
contrary to \eqref{71}. For this reason, 
$$
\tau_ku\to\beta\text{ in }L^\Phi(\Omega_0)\text{ as }k\to+\infty,
$$
and therefore, since $u$ is odd in the variable $x$ we conclude that $u\in\Gamma_\Phi^o(\beta)$. This fact combined with \eqref{71} produces that $I(u)=c_{\Phi}^o(\beta)$. Now assumptions $(A_2)$-$(A_3)$ allow us to proceed as in the proof of Lemma \ref{R3} to find a function $v\in\Gamma_\Phi^o(\beta)$ dependent on $u$ such that $v\in K_{\Phi}^o(\beta)$, and the proof is over. 
\end{proof}

We now finish this subsection by proving Theorem \ref{T6} as follows.

\noindent{\textbf{Proof of Theorem \ref{T6}.}}\\
Our proof follows the method developed in \cite{Alves1} and we will do it in detail for the reader’s convenience. To begin with, thanks to Lemma \ref{R8} we can take $u\in K_{\Phi}^o(\beta)$. Here we will first show that
\begin{equation}\label{76}
	\iint_{\Omega_0}\left(\phi\left(|\nabla u|\right)\nabla u\nabla \psi+A(x,y)V'(u)\psi\right)dxdy\geq 0,~\text{ for all }~\psi\in C^{\infty}_0(\mathbb{R}^2),
\end{equation}
which will guarantee that  
\begin{equation*}
	\iint_{\Omega_0}\left(\phi\left(|\nabla u|\right)\nabla u\nabla \psi+A(x,y)V'(u)\psi\right)dxdy= 0,~\text{ for all }~\psi\in C^{\infty}_0(\mathbb{R}^2),
\end{equation*}
implying that $u$ is a weak solution of $(PDE)$.

In what follows, for each $\psi\in C^{\infty}_0(\mathbb{R}^2)$, we will use the fact that $\psi(x,y)=\psi_o(x,y)+\psi_e(x,y)$, where 
$$
\psi_e(x,y)=\frac{\psi(x,y)+\psi(-x,y)}{2}~~~\text{and}~~~\psi_o(x,y)=\frac{\psi(x,y)-\psi(-x,y)}{2}.
$$ 
In addition to these functions, let us consider for $t>0$ the function 
$$
\tilde{\varphi}_t(x,y)=\max\left\{-\beta,\min\{\beta,\varphi_t(x,y)\}\right\}, \quad (x,y) \in \Omega,
$$ 
where 
$$
\varphi_t(x,y)=\left\{\begin{array}{lll}
	u(x,y)+t\psi_o(x,y),&\mbox{if}\quad x\geq 0\text{ and }u(x,y)+t\psi_o(x,y)\geq0,\\
	-u(x,y)-t\psi_o(x,y), &\mbox{if}\quad x\geq 0\text{ and }u(x,y)+t\psi_o(x,y)\leq0, \\-\varphi_t(-x,y),&\mbox{if}\quad x<0.
\end{array}\right.
$$ 
Now, a direct computation shows that $\tilde{\varphi}_t\in \Gamma_{\Phi}^o(\beta)$. Then $(V_6)$ together with $\tilde{\varphi}_t$ yields that
\begin{equation}\label{73}
	I(u+t\psi_o)=I(\varphi_t)\geq I(\tilde{\varphi}_t)\geq c_\Phi^o(\beta)=I(u).		
\end{equation}
On the other hand, by Lemma \ref{A5}-(b),    
\begin{equation}\label{74}
	\begin{split}	
		I(u+t\psi)-I(u+t\psi_o)\geq t\sum_{j\in\mathbb{Z}}\iint_{\Omega_j}&\phi(|\nabla(u+t\psi_o)|)\nabla u\nabla \psi_edxdy\\&+t^2\sum_{j\in\mathbb{Z}}\iint_{\Omega_j}\phi(|\nabla(u+t\psi_o)|)\nabla \psi_o\nabla \psi_edxdy\\&+\sum_{j\in\mathbb{Z}}\iint_{\Omega_j}A(x,y)\left(V(u+t\psi)-V(u+t\psi_o)\right)dxdy.
	\end{split}	
\end{equation}	
Since the functions $\phi(|\nabla (u+t\psi_o)|)\nabla u\nabla \psi_e$ and $\phi(|\nabla (u+t\psi_o)|)\nabla \psi_o\nabla \psi_e$ are odd in the variable $x$, then it is easily seen that   
\begin{equation}\label{75}
	\sum_{j\in\mathbb{Z}}\iint_{\Omega_j}\phi(|\nabla(u+t\psi_o)|)\nabla u\nabla \psi_edxdy=\sum_{j\in\mathbb{Z}}\iint_{\Omega_j}\phi(|\nabla(u+t\psi_o)|)\nabla \psi_o\nabla \psi_edxdy=0,
\end{equation} 
and so, from \eqref{73}-\eqref{75}, 
\begin{equation*}
	I(u+t\psi)-I(u)\geq\sum_{j\in\mathbb{Z}}\iint_{\Omega_j}A(x,y)\left(V(u+t\psi)-V(u+t\psi_o)\right)dxdy.
\end{equation*}
Consequently, as $A(x,y)V'(u)\psi_e$ is odd in the variable $x$, one gets  
\begin{equation}\label{151}
	\begin{split}
		\iint_{\Omega}(\phi(|\nabla u|)&\nabla u\nabla\psi+A(x,y)V'(u)\psi)dxdy=\lim_{t\rightarrow 0^+}\frac{I(u+t\psi)-I(u)}{t}\\&\geq\lim_{t\rightarrow 0^+}\sum_{j\in\mathbb{Z}}\iint_{\Omega_j}A(x,y)\frac{V(u+t\psi)-V(u+t\psi_o)}{t}dxdy\\&\geq \sum_{j\in\mathbb{Z}}\iint_{\Omega_j}A(x,y)V'(u)(\psi-\psi_o)dxdy=\sum_{j\in\mathbb{Z}}\iint_{\Omega_j}A(x,y)V'(u)\psi_edxdy=0, 
	\end{split}	
\end{equation} 
showing that the inequality \eqref{76} occurs for every $\psi\in C^{\infty}_0(\mathbb{R}^2)$. Finally, slightly varying the same ideas discussed in the Subsection \ref{Sb1}, we can conclude the proof of Theorem \ref{T6}.   \hspace{0.8 cm} $\Box$

%%%%%%%%%%%%%%%%%%%%%%%%%%%%%%%%%%%%%%%%%%%%%%%%%%%%%%

\section{Heteroclinic solution of the prescribed curvature equation}\label{S3}

Throughout this section, we adapt for our problem the approach explored in \cite{Alves2} to find solutions that are periodic in the variable $y$ and heteroclinic in $x$ from $\alpha$ to $\beta$ to the prescribed mean curvature equation $(E)$. In this section, we will return to indexing $\alpha$ and $\beta$ in $V$, denoted as $V_{\alpha,\beta}$, because here our focus is on choosing suitable values of $\alpha$ and $\beta$. Since the ideas are so close to those of \cite{Alves2}, the presentation will be brief. In the following, we consider for each $L>0$ the quasilinear equation  
$$
-\Delta_{\Phi_L}u+A(\epsilon x,y)V'_{\alpha,\beta}(u)=0~~\text{in}~~\mathbb{R}^2,\eqno{(E)_L}
$$
with $V_{\alpha,\beta} \in \mathcal{V}$, where $\Phi_L:\mathbb{R}\rightarrow[0,+\infty)$ is an $N$-function of the form
$$
\Phi_L(t)=\int_{0}^{|t|}\phi_L(s)sds,
$$
where $\phi_L(t)=\varphi_L(t^2)$ and  $\varphi_L$ is defined by 
$$
\varphi_L(t)=\left\{\begin{array}{lll}
	\dfrac{1}{\sqrt{1+t}},&\mbox{if}\quad t\in[0,L],
	\\
	x_L(t-L-1)^2+y_L,&\mbox{if}\quad t\in[L,L+1],
	\\
	y_L,&\mbox{if}\quad t\in[L+1,+\infty),
\end{array}\right.
$$
with
$$
x_L=\frac{\sqrt{1+L}}{4(1+L)^2}~\text{ and }~y_L=(4L+3)x_L.
$$

We point out that the main purpose of this section is to use the arguments of Sect. \ref{S2} to investigate the existence of a heteroclinic solution $u_{\alpha,\beta}$ from $\alpha$ to $\beta$ for $(E)_L$ that satisfies 
\begin{equation} \label{EQTL} 
\|\nabla u_{\alpha,\beta}\|_{L^{\infty}(\mathbb{R}^2)}\leq \sqrt{L},
\end{equation} 
because this inequality implies that $u_{\alpha,\beta}$ is a heteroclinic solution from $\alpha$ to $\beta$ for $(E)$. Here, we will prove that the inequality \eqref{EQTL} holds when $\max\{|\alpha|,|\beta|\}$ is small enough. In order to do that, a control involving the roots $\alpha$ and $\beta$ of $V_{\alpha,\beta}$ is necessary, and at this point the condition $(V_4)$ applies an important rule in our argument.

The next result is about functions $\phi_L$ and $\Phi_L$, which makes it clear that $\Phi_L$ is an $N$-function.

\begin{lemma}\label{R9}
	For each $L>0$, the functions $\phi_L$ and $\Phi_L$ have the following properties:
	\begin{itemize}
		\item [(a)] $\phi_L$ is $C^1$.
		\item [(b)] $y_L\leq \phi_L(t)\leq 1$ for all $t\geq 0$.
		\item [(c)] $\frac{y_L}{2}t^2\leq \Phi_L(t)\leq \frac{1}{2}t^2$ for any $t\in\mathbb{R}$.
		\item [(d)] $\Phi_L$ is a convex function.
		\item [(e)] $\left(\phi_L(t)t\right)'>0$ for all $t>0$.
	\end{itemize}
\end{lemma}
\begin{proof}
	The argument follow the same lines as the proof of \cite[Lemma 3.1]{Alves2}.
\end{proof}

We would like to point out that our focus now is on examining if the $N$-function $\Phi_L$ is in the settings of Sect. \ref{S2}, that is, if $\phi_L$ satisfies conditions $(\phi_1)$-$(\phi_3)$. Indeed, it is clear that by Lemma \ref{R9}-(e) $\phi_L$ checks $(\phi_1)$, and by Lemma \ref{R9}-(b), it checks $(\phi_3)$ with $q=2$. Moreover, with direct computations one can get that there are real numbers $m_L, l_L>1$ such that $l_L\leq m_L$ and
$$
l_L-1\leq \frac{(\phi_L(t)t)'}{\phi_L(t)}\leq m_L-1~\text{ for any }~t\geq 0,
$$
from which it follows that $\phi_L$ verifies $(\phi_2)$.
As a direct consequence the N-functions $\Phi_L$ and $\tilde{\Phi}_L$ satisfy $\Delta_2$-condition, where $\tilde{\Phi}_L$ is the complementary function associated with $\Phi_L$, which ensures that the space $L^{\Phi_L}$ is reflexive (see for instance Appendix \ref{A}). Actually, the study made in \cite[Lemma 3.2]{Alves2} shows that the space $L^{\Phi_L}$ is exactly $L^2$ space and the norm of $L^{\Phi_L}$ is equivalent to the norm of $L^2$.

Now, assuming for a moment that function $A$ belongs to Class 1 or 2, $\epsilon=1$, and that the potential $V_{\alpha,\beta}\in \mathcal{V}$, the same arguments from Subsections \ref{Sb1} and \ref{Sb2} guarantee that there exist a periodic function $u_{\alpha,\beta}:\mathbb{R}^2\to\mathbb{R}$ on the variable $y$ such that $u_{\alpha,\beta}\in K_{\Phi_L}(\alpha,\beta)$, where
$$
K_{\Phi_L}(\alpha,\beta)=\left\{w\!\in\!\Gamma_{\Phi_L}(\alpha,\beta)\!:\!I(w)\!=\!c_{\Phi_L}\!(\alpha,\beta),~w(x,0)=w(x,1)\text{ in }\mathbb{R},~\alpha\leq w\leq\beta\text{ a.e. on }\Omega\right\}.
$$
Moreover, $u_{\alpha,\beta}$ is a weak solution of equation $(E)_L$ with $\epsilon=1$ in $C^{1,\gamma}_{\text{loc}}(\mathbb{R}^2)$ for some $\gamma\in(0,1)$ that is heteroclinic in $x$ from $\alpha$ to $\beta$. The next lemma is crucial to guarantee the existence of a heteroclinic solution for $(E)$ and its proof is inspired by the proof of \cite[Lemma 3.3]{Alves2}.

\begin{lemma}\label{R10}
There exists $\delta>0$ such that for each $V_{\alpha,\beta} \in \mathcal{V}$ with $\max\left\{|\alpha|,|\beta|\right\}<\delta$, the heteroclinic solution $u_{\alpha,\beta}\in K_{\Phi_L}(\alpha,\beta)$ obtained in the previous section satisfies the estimate below  
		$$
	\|u_{\alpha,\beta}\|_{C^1(B_{1}(z))}<\sqrt{L}, \quad \forall z\in\mathbb{R}^2,
	$$
	where $B_1(z)$ denotes the ball in $\mathbb{R}^2$ of center $z$ and radius 1.
\end{lemma}
\begin{proof}
If the lemma does not hold, then for each $n \in \mathbb{N}$  there is a potential $V_{\alpha_n,\beta_n}\in \mathcal{V}$ with $
	\max\left\{|\alpha_n|,|\beta_n|\right\}<n^{-1}$ and $z_n=(r_n,s_n)\subset\mathbb{R}^2$ such that  the heteroclinic solution $u_{\alpha_n,\beta_n}\in K_{\Phi_L}(\alpha_n,\beta_n)$ satisfies 
\begin{equation}\label{**}
	\|u_{\alpha_n,\beta_n}\|_{C^1(B_1(r_n,s_n))}\geq \sqrt{L},~~~\forall n\in\mathbb{N}.
\end{equation} 
Now, we note that for each $n\in\mathbb{N}$ the function $\tilde{u}_n$ defined by 
$$
\tilde{u}_n(x,y)=u_{\alpha_n,\beta_n}(x+r_n,y+s_n)~\text{ for }~(x,y)\in\mathbb{R}^2
$$
is a weak solution of the quasilinear equation
$$
-\Delta_{\Phi_L}u+B_n(x,y)=0\text{ in }\mathbb{R}^2,
$$
where 
$$
B_n(x,y)=A(x+r_n,y+s_n)V'_{\alpha_n,\beta_n}(u_{\alpha_n,\beta_n}(x,y)).
$$
Furthermore, the condition $(V_4)$ is crucial to show the following inequality 
\begin{equation}\label{NI}
	|B_n(x,y)|\leq M\|A\|_{L^{\infty}(\mathbb{R}^2)}~~\forall(x,y)\in\mathbb{R}^2\text{ and }\forall n\geq n_0,
\end{equation}
where $n_0\in\mathbb{N}$ and $M$ is a positive constant independent of $n$. Indeed, let us first observe that the condition $\max\left\{|\alpha_n|,|\beta_n|\right\}\to 0$ as $n\to+\infty$ implies that there exists $n_0\in\mathbb{N}$ such that $\alpha_n , \beta_n \in (-1,1)$ for all $n\geq n_0$. Since 
$$
\alpha_n\leq u_{\alpha_n,\beta_n}(x,y)\leq\beta_n,~~ \forall (x,y)\in\mathbb{R}^2~~\text{and}~~\forall n\in\mathbb{N},
$$ 
by condition $(V_4)$, it follows that there exists $M=M(1)>0$ such that  
$$
\left|V_{\alpha_n,\beta_n}'\left(u_{\alpha_n,\beta_n}(x,y)\right)\right|\leq M,~~\forall (x,y)\in\mathbb{R}^2~~\text{and}~~ \forall n\geq n_0,
$$
which is sufficient to guarantee estimate \eqref{NI}. Therefore, the elliptic regularity theory found in \cite[Theorem 1.7]{Lieberman} implies that $\tilde{u}_n\in C^{1,\gamma_0}_{\text{loc}}(\mathbb{R}^2)$, for some $\gamma_0\in(0,1)$, and that there is a positive constant $R$ independent of $n$ verifying
$$
\|\tilde{u}_n\|_{C^{1,\gamma_0}_{\text{loc}}(\mathbb{R}^2)}\leq R~~\forall~n\in\mathbb{N}.
$$
The above estimate allows us to use Arzelà–Ascoli Theorem to find $u\in C^1(B_1(0))$ and a subsequence of $(\tilde{u}_n)$, still denoted by $(\tilde{u}_n)$, such that
$$
\tilde{u}_n\to u~~\text{ in }~~C^1(B_1(0)).
$$
Now since $\|u_{\alpha_n,\beta_n}\|_{L^{\infty}(\mathbb{R}^2)}$ tends to zero as $n\to +\infty$, we obtain that $u=0$ on $B_1(0)$, and so, 
$$
\|\tilde{u}_n\|_{C^1\left(B_1(0)\right)}<\sqrt{L},~~\forall~n\geq n_0,
$$
for some $n_0\in\mathbb{N}$. Therefore, 
$$
\|u_{\alpha_n,\beta_n}\|_{C^1(B_1(r_n,s_n))}<\sqrt{L},~~\forall~n\geq n_0,
$$ 
which contradicts \eqref{**}, and the proof is completed.
\end{proof}

We are finally ready to prove the Theorem \ref{T1}.

\noindent{\textbf{Proof of Theorem \ref{T1}.}}\\ 
To begin with, we claim that given $L>0$ there exists $\delta>0$ such that if $\max\{|\alpha|,|\beta|\}\in(0,\delta)$ we have
\begin{equation}\label{78}
	\|\nabla u_{\alpha,\beta}\|_{L^{\infty}(\mathbb{R}^2)}\leq \sqrt{L}~\text{ for all }~u_{\alpha,\beta}\in K_{\Phi_L}(\alpha,\beta).
\end{equation}
Indeed, for each $(x,y)\in\mathbb{R}^2$ we can choose $z\in\mathbb{R}^2$ verifying $(x,y)\in B_1(z)$. Thanks to Lemma \ref{R10}, there is $\delta>0$ such that for each  pair $(\alpha,\beta)$ of real numbers with $\alpha<\beta$ and $\max\{|\alpha|,|\beta|\}\in(0,\delta)$ one has
$$
\|u_{\alpha,\beta}\|_{C^1(B_{1}(z))}<\sqrt{L}
$$
whenever $u_{\alpha,\beta}\in K_{\Phi_L}(\alpha,\beta)$. Now, from the arbitrariness of $(x,y)$ it is easy to see that our claim is established. Therefore, the estimate \eqref{78} ensures  $u_{\alpha,\beta}$ is a heteroclinic solution of $(E)$. To complete the proof, the fact that $V_{\alpha,\beta}\in C^2(\mathbb{R},\mathbb{R})$ combined with the assumptions $(V_2)$-$(V_3)$ yields that there are $\lambda,d_1,d_2>0$ such that 
$$
|V'_{\alpha,\beta}(t)|\leq d_1|t-\alpha|,~~\forall t\in[\alpha-\lambda,\alpha+\lambda]
$$
and 
$$
|V'_{\alpha,\beta}(t)|\leq d_2|t-\beta|,~~\forall t\in [\beta-\lambda,\beta+\lambda].
$$
Thus, by Lemma \ref{R9}-(b),
$$
|V'_{\alpha,\beta}(t)|\leq \frac{d_1}{y_L}\phi_L(|t-\alpha|)|t-\alpha|,~~\forall t\in[\alpha-\lambda,\alpha+\lambda]
$$
and 
$$
|V'_{\alpha,\beta}(t)|\leq \frac{d_2}{y_L}\phi_L(|t-\beta|)|t-\beta|,~~\forall t\in [\beta-\lambda,\beta+\lambda].
$$
Consequently, $V_{\alpha,\beta}$ satisfies $(V_7)$ with $\phi_L$, and so,  proceeding as in the proof of Lemma \ref{R4a}, we see that $u_{\alpha,\beta}$ verifies
$$
\alpha<u_{\alpha,\beta}(x,y)<\beta\text{ for all }(x,y)\in\mathbb{R}^2,
$$
which is the desired conclusion.\hspace{10 cm} $\square$

Now, let us assume that $A$ belongs to Class 3 and that $V_{\alpha,\beta}\in\mathcal{V}$. Considering the set $K_{\epsilon,\Phi_L}(\alpha,\beta)$ as in \eqref{80}, then we can argue similarly to the proof of Theorem \ref{T5} to obtain that $K_{\epsilon,\Phi_L}(\alpha,\beta)\neq\emptyset$ whenever $\epsilon>0$ is small. With everything, proceeding analogously as in the proof of Lemma \ref{R10} we get the following result.

\begin{lemma}\label{R11}
There exist $\delta>0$ independent of $\epsilon>0$ and $\epsilon_0>0$ such that for each $\epsilon\in(0,\epsilon_0)$ and $V_{\alpha,\beta} \in \mathcal{V}$ with $\max\{|\alpha|,|\beta|\}\in(0,\delta)$, the heteroclinic solution $u_{\epsilon,\alpha,\beta}\in K_{\epsilon,\Phi_L}(\alpha,\beta)$ satisfies the following estimate  
	$$
	\|u_{\epsilon,\alpha,\beta}\|_{C^1(B_{1}(z))}<\sqrt{L}, \quad \forall z\in\mathbb{R}^2.
	$$
\end{lemma}

We now present the proof of Theorem \ref{T2}.

\noindent{\textbf{Proof of Theorem \ref{T2}.}}\\
The proof can be done via a comparison argument like that of the proof Theorem \ref{T1} and we omit it here. Details are left to the reader. \hspace{8,3 cm} $\square$

For the final exhibition of these ideas, let's assume that the function $A$ belongs to Class 4, $\alpha=-\beta$, $\epsilon=1$ and that each $V_{-\beta,\beta}\in\mathcal{V} \cap C^2(\mathbb{R},\mathbb{R})$ also satisfies the conditions $(V_5)$ and $(V_6)$. We want to point out that condition $(V_5)$ implies that the potential $V_{-\beta,\beta}$ satisfies $(V_8)$ with $\Phi_L$. In fact, note that by $(V_5)$ there are $\rho>0$ and $\theta\in (0,\frac{\beta}{2})$ such that
\begin{equation}\label{***}
	\rho|t-\beta|^2\leq V_{-\beta,\beta}(t),~\forall t\in(\beta-\theta,\beta+\theta),
\end{equation}
from which it follows by Lemma \ref{R9}-(c) and \eqref{***},
$$
2\rho\Phi_L(|t-\beta|)\leq V_{-\beta,\beta}(t),~\forall t\in(\beta-\theta,\beta+\theta).	
$$	
Consequently, the argument of Subsection \ref{Sb5} shows that for each $L>0$ the set $K_{\Phi_L}^o(\beta)$ is not empty, where $K_{\Phi_L}^o(\beta)$ is given as in \eqref{79}. We would like to remind here that each element of $K_{\Phi}^o(\beta)$ can be seen as a function on $\mathbb{R}^2$ being periodic in the variable $y$. Moreover, if $u_\beta\in K_{\Phi_L}^o(\beta)$, then $u_\beta$ is a weak solution for $(E)_L$ with $\epsilon=1$ in $C^{1,\gamma}_{\text{loc}}(\mathbb{R}^2)$ for some $\gamma\in(0,1)$, odd and heteroclinic from $-\beta$ to $\beta$ in $x$ satisfying
$$
0\leq u_\beta(x,y)\leq\beta\text{ for all }(x,y)\in\mathbb{R}_+\times\mathbb{R}.
$$
Now, the following result is a similar version of Lemma \ref{R10} and is proved in an essentially identical fashion, which will play an analogous role to that developed in Theorem \ref{T1} in the present setting.

\begin{lemma}\label{R12}
There exists $\delta>0$ such that for each $V_{-\beta,\beta}\in\mathcal{V} \cap C^2(\mathbb{R},\mathbb{R})$ satisfying $(V_5)$-$(V_6)$ with $\beta\in(0,\delta)$, the heteroclinic solution $u_{\beta}\in K_{\Phi_L}^o(\beta)$ satisfies 
	$$
	\|u_{\beta}\|_{C^1(B_{1}(z))}<\sqrt{L}, \quad \forall z\in\mathbb{R}^2.
	$$
\end{lemma}

Finally, to conclude this section, we prove Theorem \ref{T3} using the framework discussed above.

\noindent{\textbf{Proof of Theorem \ref{T3}.}}\\
The proof is established using Lemma \ref{R12} and arguing as in the proof of Theorem \ref{T1}. Detailed verification is left to the reader. \hspace{10,0 cm} $\square$

%%%%%%%%%%%%%%%%%%%%%%%%%%%%%%%%%%%%%%%%%%%%%%%%%%%%%%%%

%%%%%%%%%%%%%%%%%%%%%%%%%%%%%%%%%%%%%%%%%%%%%%%%%%%%%%

\appendix
\section{Orlicz and Orlicz-Sobolev Spaces}\label{A}

%%%%%%%%%%%%%%%%%%%%%%%%%%%%%%%%%%%%%%%%%%%%%%%%%%%%%%

%%%%%%%%%%%%%%%%%%%%%%%%%%%%%%%%%%%%%%%%%%%%%%%%%%%%%%

In the years 1932 and 1936 Orlicz considered and investigated in \cite{Orlicz1,Orlicz2} the following class of functions
$$
L^{\Phi}(\mathcal{O})=\left\{u\in L^1_{\text{loc}}(\mathcal{O}):~\int_\mathcal{O} \Phi\left(\frac{|u|}{\lambda}\right) dx<+\infty~\text{for some}~\lambda>0\right\},
$$ 
where $\mathcal{O}$ is an open set of $\mathbb{R}^N$ with $N \geq 1$ and the function $\Phi:\mathbb{R}\rightarrow [0,+\infty)$ has the following properties:
\begin{itemize}
	\item[(a)] $\Phi$ is continuous, convex and even,
	\item[(b)] $\Phi(t)=0$ if and only if $t=0$,
	\item[(c)] $\displaystyle \lim_{t\rightarrow 0}\dfrac{\Phi(t)}{t}=0$ and $\displaystyle \lim_{t\rightarrow +\infty}\dfrac{\Phi(t)}{t}=+\infty$.
\end{itemize}
In these configurations, we say that $\Phi$ is an \textbf{$N$-function}. In Orlicz's paper \cite{Orlicz1}, he also introduced an additional condition on the function $\Phi$ the so called \textbf{$\Delta_2$-condition} ($\Phi\in\Delta_2$ for short) which says that there exist constants $C>0$ and $t_{0}\geq0$ such that 
$$
\Phi(2t)\leq C\Phi(t),~~\forall t \geq t_0. \eqno{(\Delta_2)}
$$ 
Below are some examples of $N$-functions that satisfy $(\Delta_2)$ with $t_0=0$:
\begin{itemize}
	\item[1.] $\Phi_1(t)=\dfrac{|t|^p}{p}$ with $p\in(1,+\infty)$,
	\item[2.] $\Phi_2(t)=\dfrac{|t|^p}{p}+\dfrac{|t|^q}{q}$ for $p<q$ and $p,q\in(1,+\infty)$,
	\item[3.]  $\Phi_3(t)=|t|^p\ln(1+|t|)$, where $p\in(1,+\infty)$,
	\item[4.] $\Phi_4(t)=(1+t^2)^{\gamma}-1$ with $\gamma>1$,
	\item[5.] $\Phi_5(t)=(\sqrt{1+t^2}-1)^\gamma$ for $\gamma>1$,
	\item [6.] $\Phi_6(t)=\displaystyle\int_{0}^{t}s^{1-\gamma}(\sinh^{-1}s)^{\beta}ds$ with $0\leq\gamma< 1$ and $\beta>0$.
\end{itemize}

The following Minkowski functional
$$
\|u\|_{L^{\Phi}(\mathcal{O})}=\inf\left\{\lambda>0:\int_\mathcal{O} \Phi\left(\frac{|u|}{\lambda}\right) dx\leq 1\right\},~~u\in L^{\Phi}(\mathcal{O}),
$$
was introduced by Luxemburg in his thesis \cite{Luxemburg} and the reader can verify that $\|u\|_{L^{\Phi}(\mathcal{O})}$ defines a norm on $L^\Phi(\mathcal{O})$, which is called the \textbf{Luxemburg norm} over $\mathcal{O}$. A direct computation ensures that $L^{\Phi}(\mathcal{O})$ endowed with the  Luxemburg norm is a Banach space. To illustrate this phenomenon in particular cases, consider $1<p<+\infty$ and $\Phi_p(t)=\frac{|t|^p}{p}$ to see that
$$
L^{\Phi_p}(\mathcal{O})=L^{p}(\mathcal{O})~~\text{ and }~~ \|u\|_{L^{\Phi_p}(\mathcal{O})}=p^{-\frac{1}{p}}\|u\|_{L^{p}(\mathcal{O})},
$$
where $\|\cdot\|_{L^p(\mathcal{O})}$ is the usual norm of space $L^p(\mathcal{O})$. This fact makes it clear that the spaces $L^{\Phi}(\mathcal{O})$ are more general than Lebesgue's $L^p$ spaces. However, these spaces may possess peculiar properties that do not occur in ordinary Lebesgue's spaces. Now let us note as a point of interest that if $|\mathcal{O}|<+\infty$, then a direct computation shows that the embedding 
\begin{equation*}
	L^\Phi(\mathcal{O})\hookrightarrow L^1(\mathcal{O})	
\end{equation*} 
is continuous. In other words, this embedding is equivalent to the simple containment $	L^\Phi(\mathcal{O})\subset L^1(\mathcal{O})$ in which some topological properties are preserved. The $L^\Phi(\mathcal{O})$ space became known in the literature as \textbf{Orlicz space} over $\mathcal{O}$. In \cite{Orlicz1}, Orlicz restricted himself to the case $\Phi \in \Delta_2$ and in this case we may write the Orlicz spaces as follows
$$
L^{\Phi}(\mathcal{O})=\left\{u\in L^1_{\text{loc}}(\mathcal{O}):~\int_\mathcal{O} \Phi(|u|) dx<+\infty\right\}.
$$ 
Furthermore, still under $\Delta_2$-condition we have the following fact 
\begin{equation}\label{NT0}
	u_n\rightarrow u~~ \text{in}~~L^{\Phi}(\mathcal{O})\Leftrightarrow \int_{\mathcal{O}}\Phi(|u_n-u|)dx\rightarrow 0.
\end{equation}
It is important to mention here that in general these facts do not occur when $\Phi$ does not satisfy $(\Delta_2)$. Orlicz was the first to investigate this case, that is, without the $\Delta_2$-condition, in his paper \cite{Orlicz2} in the year 1936.

A natural generalization of Sobolev spaces $W^{1,p}(\mathcal{O})$ is the following Banach space associated with an $N$-function $\Phi$ 
$$
W^{1,\Phi}(\mathcal{O})=\left\{u\in L^{\Phi}(\mathcal{O}):~\frac{\partial u}{\partial x_i}=u_{x_i}\in L^{\Phi}(\mathcal{O}),~i=1,...,N\right\},
$$ 
equipped with the norm
$$
\|u\|_{W^{1,\Phi}(\mathcal{O})}=\|\nabla u\|_{L^{\Phi}(\mathcal{O})}+\|u\|_{L^{\Phi}(\mathcal{O})},
$$ 
where $\nabla u=(u_{x_1},...,u_{x_N})$. Space $W^{1,\Phi}(\mathcal{O})$ is usually called  \textbf{Orlicz-Sobolev space} associated with $\Phi$ over $\mathcal{O}$. It is impossible for the authors to give an exhaustive account here of the vast literature devoted to Orlicz-type spaces. For a quite comprehensive account of this topic, the interested reader might start by referring to \cite{Krasnosels'kii,Rao} and the bibliography therein. 

We are going to now consider some results that will be widely used in the development of this work. Initially we observe that given an $N$-function $\Phi$, the {\it complementary function} $\tilde{\Phi}$ associated with $\Phi$ is defined by Legendre’s transformation
$$
\tilde{\Phi}(s)=\max_{t\geq 0}\left\{st-\Phi(t)\right\}\text{ for } s\geq 0.
$$ 
The function  $\tilde{\Phi}$ is also an $N$-function, and moreover, the functions $\Phi$ and $\tilde{\Phi}$ are complementary each other. An interesting example of such complementary functions are
$$
\Phi(t)=\frac{|t|^p}{p}~\text{ and }~\tilde{\Phi}(t)=\frac{|t|^q}{q}~\text{ with }~1<p<+\infty~\text{ and }~\frac{1}{p}+\frac{1}{q}=1.
$$
An important result of the theory of Orlicz spaces is the following.
\begin{lemma}\label{A1}
	The space $L^{\Phi}(\mathcal{O})$ is reflexive if, and only if, $\Phi$ and $\tilde{\Phi}$ satisfy the $\Delta_2$-condition. 
\end{lemma}
\begin{proof}
	See for instance \cite[Ch. IV, Theorem 10]{Rao}.
\end{proof}

Before ending this section, we would like to point out that a function $\Phi$ of the form \eqref{*} satisfying $(\phi_1)$ and $(\phi_2)$ is an $N$-function. Moreover, the reader can verify that examples 1 to 6 listed above are models of $N$-functions of the type \eqref{*} checking $(\phi_1)$-$(\phi_2)$. Next, we list some lemmas about $N$-functions of the form \eqref{*} that will be used frequently in this work.

\begin{lemma}\label{A2}
	Let $\Phi$ be an $N$-function of the form \eqref{*} satisfying $(\phi_1)$-$(\phi_2)$. Set 
	$$
	\xi_0(t)=\min\{t^{l},t^{m}\}~~\text{and}~~\xi_1(t)=\max\{t^{l},t^{m}\},~~\forall t\geq 0.
	$$ 
	Then $\Phi$ satisfies 
	$$
	\xi_0(t)\Phi(s)\leq\Phi(st)\leq\xi_1(t)\Phi(s),~~\forall s,t\geq0.
	$$
	In particular, $\Phi\in\Delta_2$.
\end{lemma}
\begin{proof}
Let's first show that condition $(\phi_2)$ leads to 
\begin{equation}\label{0}
	l\leq \dfrac{\phi(t)t^{2}}{\Phi(t)}\leq m,~~\forall t>0.
\end{equation}
Indeed, by $(\phi_2)$ we can write
$$
l\phi(t)\leq \left(\phi(t)t\right)'+\phi(t)\leq m\phi(t),~~\forall t>0,
$$
from which it follows that
$$
l\phi(t)t\leq\left(\phi(t)t^2\right)'\leq m\phi(t)t,~~\forall t>0.
$$
Therefore, integrating the last inequality we get \eqref{0}, as we wanted. Finally, due to estimation \eqref{0} the proof becomes similar to that given in \cite[Lemma 2.1]{Fukagai}. The details are left to the reader.
\end{proof}

\begin{lemma}\label{A3}
	Let $\Phi$ be an $N$-function of the form \eqref{*} satisfying $(\phi_1)$-$(\phi_2)$. Then, $\tilde{\Phi}$ satisfies the $\Delta_2$-condition. 
\end{lemma}
\begin{proof}
 See \cite[Lemma 2.7 ]{Fukagai}.
\end{proof}

\begin{lemma}\label{A4}
	If $\Phi$ is an $N$-function of the form \eqref{*} satisfying $(\phi_1)$-$(\phi_2)$, then the spaces $L^{\Phi}(\mathcal{O})$ and $W^{1,\Phi}(\mathcal{O})$ are reflexive.
\end{lemma}
\begin{proof}
	The proof follows directly from Lemmas \ref{A1}, \ref{A2} and \ref{A3}.
\end{proof}

For our study, it is useful to consider the following lemma whose proof is left for the reader to verify.

\begin{lemma}\label{A5}
	Let $\Phi$ be an $N$-function of the type \eqref{*} satisfying $(\phi_1)$-$(\phi_2)$. Then the following inequalities hold
	\begin{itemize}
		\item[(a)]  $\Phi(|a+b|)\leq 2^m\left(\Phi(|a|)+\Phi(|b|)\right)$ for all  $a,b\in\mathbb{R}$.
		\item[(b)] $\phi(|z|)z.(w-z)\leq\Phi(|w|)-\Phi(|z|)$ for all $w,z\in\mathbb{R}^N$ with $z\neq0$, where ``$.$" denotes the usual inner product in $\mathbb{R}^N$.
	\end{itemize}
\end{lemma}

%%%%%%%%%%%%%%%%%%%%%%%%%%%%%%%%%%%%%%%%%%%%%%%%%%%%%%%

\vspace{0.2 cm}
\noindent {\bf Acknowledgments:} The authors would like to thank to the anonymous referee for their invaluable suggestions, which significantly contributed to the enhancement of this paper.

\end{document}